\documentclass{amsart}
\usepackage{amssymb}
\usepackage{hyperref}
\usepackage{tikz-cd}
\usetikzlibrary{decorations.pathmorphing}
\newtheorem{theorem}{Theorem}[section]

\newtheorem{fact}[theorem]{Fact}
\newtheorem{lemma}[theorem]{Lemma}
\newtheorem{corollary}[theorem]{Corollary}
\newtheorem{claim}[theorem]{Claim}
\newtheorem{proposition}[theorem]{Proposition}

\newtheorem{question}[theorem]{Question}
\newtheorem{introtheorem}{Theorem}

\theoremstyle{definition}
\newtheorem{definition}[theorem]{Definition}
\newtheorem{example}[theorem]{Example}
\newtheorem{non-example}[theorem]{Non-Example}

\newtheorem{notation}[theorem]{Notation}

\newtheorem{assumption}[theorem]{Assumption}

\theoremstyle{remark}
\newtheorem{remark}[theorem]{Remark}

\newcommand{\acl}{\operatorname{acl}}

\newcommand{\tp}{\operatorname{tp}}

\begin{document}

\title{Reconstructing Abelian Varieties via Model Theory}
\author{Benjamin Castle and Assaf Hasson}
\address{Department of Mathematics, University of Illinois Urbana-Champaign}
\email{btcastl2@illinois.edu}
\address{Department of Mathematics, Ben Gurion University of the Negev, Be'er Sehva, Israel}
\email{hassonas@math.bgu.ac.il}

\begin{abstract}
    In 2012, Zilber used model-theoretic techniques to show that a curve of high genus over an algebraically closed field is determined by its Jacobian (viewed only as an abstract group with a distinguished subset for an image of the curve). In this paper, we consider an analogous problem for arbitrary (semi)abelian varieties $A$ over algebraically closed fields $K$ with a distinguished subvariety $V$. Our main result characterizes when the data $(A(K),+,V(K))$ (as a group with distinguished subset) determines the pair $(A,V)$ in the strongest reasonable sense. 
    
    As it turns out, the situation is best understood by developing a theory of \textit{factorizations} for such pairs $(A,V)$. In the final sections of the paper, we develop such a theory and prove unique factorization theorems (one for abelian varieties and a weaker one for semi-abelian varieties). In this language, the main theorem mentioned above (in the abelian case) says that the pair $(A,V)$ is determined by the data $(A(K),+,V(K))$ precisely when $(A,V)$ is simple and $0<\dim(V)<\dim(A)$.
\end{abstract}

\maketitle

\section{Introduction}
\subsection{History and Motivation}
	In \cite{ZilJac}, Zilber uses model theoretic techniques to solve a conjecture of Bogomolov, Korotiaev and Tschinkel. Informally, he shows that one can reconstruct a projective algebraic curve of genus at least 2 over an algebraically closed field from an Abel-Jacobi embedding into its Jacobian (viewed as an abstract group).  Let us give a precise formulation of Zilber's result in a language that will be convenient for our purposes. 
	
	\begin{definition}\label{D: determines}
		Let $G$ be an algebraic group over an algebraically closed field $K$, and let $V$ be a closed irreducible subvariety of $G$. We say that $V$ \textit{geometrically determines $G$} if the following holds: let $G'$ be another algebraic group over an algebraically closed field $K'$, with a closed irreducible subvariety $V'$. Let $\phi:G(K)\rightarrow G'(K')$ be a group isomorphism satisfying $\phi(V(K))=V'(K')$. Then $\phi$ decomposes as a composition $\phi=f\circ\sigma$, where $\sigma$ is induced by a field automorphism $K\rightarrow K'$ and $f$ is a bijective isogeny of algebraic groups over $K'$.
	\end{definition}
	
	Definition \ref{D: determines} says that from the data $(G(K),\cdot,V(K))$ (viewed as a group with a distinguished subset), one is able to reconstruct the full algebraic structure of $G$ in the strongest possible sense. Note, in particular, that field isomorphisms and bijective isogenies always induce isomorphisms of abstract groups, and thus cannot be avoided. 
	
	In this language, Zilber shows:
	
	\begin{fact}\label{F: zil jac} Let $A$ be an abelian variety of dimension at least 2 over an algebraically closed field, and let $C$ be a closed irreducible curve in $A$. Assume further that:
    \begin{enumerate}
        \item $C$ is smooth and $0\in C$.
        \item $A$ is isomorphic over $C$ to its Jacobian. That is, let $C\rightarrow\operatorname{Jac}(C)$ be the Abel-Jacobi embedding determined by the point $0\in C$. Then the induced map $\operatorname{Jac}(C)\rightarrow A$ (making $C\rightarrow\operatorname{Jac}(C)\rightarrow A$ commute) is an isomorphism.
    \end{enumerate}
    Then $C$ geometrically determines $A$.
	\end{fact}
	
	The crucial input of Zilber's proof is a general model-theoretic reconstruction result for algebraically closed fields, denoted in his paper as the \textit{Rabinovich Theorem} (\cite{Ra}). In fact, the Rabinovich Theorem is a special case of a larger program of Zilber, called the \textit{Restricted Trichotomy Conjecture}. The full Restricted Trichotomy Conjecture concerns all varieties over an algebraically closed field, while Rabinovich's Theorem is the same statement for rational curves.
	
	Since Zilber's paper, significant progress emerged toward the full Restricted Trichotomy Conjecture (see \cite{hasu}, \cite{CasACF0}, \cite{CaHaYe}), culminating in a complete proof in \cite{CaHaYe}. In particular, the statement of Rabinovich's Theorem now extends to varieties of any dimension. In light of this, it is natural to consider variants of Fact \ref{F: zil jac} replacing $C$ by a higher-dimensional variety. For example, a natural question is the following:
	
	\begin{question}\label{Q: main question} Let $A$ be an abelian variety over an algebraically closed field $K$. Which irreducible closed subvarieties of $A$ geometrically determine $A$?
	\end{question}
	
	\subsection{Main Results} There are three straightforward families of negative examples to Question \ref{Q: main question} (see Corollary \ref{C: determines characterization} for details):
	
	\begin{enumerate}
		\item Suppose $V$ is contained in a proper abelian subvariety $B$ of $A$. Then $V$ does not geometrically determine $A$, essentially because an automorphism of $(A(K),+,V(K))$ (as a group with distinguished subset) can involve wild group automorphisms of the quotient $A(K)/B(K)$.
		\item On the opposite side, suppose $V$ is really contained in a proper quotient -- that is, suppose $V$ is stabilized by a proper abelian subvariety $B$ of positive dimension. Then $V$ does not geometrically determine $A$, essentially because an automorphism of $(A(K),+,V(K))$ can involve wild group automorphisms of $B(K)$.
		\item Finally, suppose $V$ is essentially a product -- that is, suppose there are proper abelian subvarieties $B_1,B_2$ of $A$ such that the sum map $B_1\times B_2\rightarrow A$ is an isogeny, and such that $V=W_1+W_2$ for some subvarieties $W_i$ of $B_i$. Then $V$ does not geometrically determine $A$, but now it is less obvious. The idea is to choose an automorphism of $(A(K),+,V(K))$ by giving the identity on $B_1(K)$ and a wild field automorphism on $B_2(K)$, and show that these cannot be patched together into a single factorization $f\circ\sigma$ of the desired form.
	\end{enumerate}
	
	Our first main result says that the above examples are the only obstructions. In fact, the proof works also for semiabelian varieties:
	
	\begin{introtheorem}\label{T: main}
		Let $K$ be an algebraically closed field. Let $G$ be a semiabelian variety over $K$. Let $V$ be a closed, irreducible subvariety of $G$. Assume that:
		\begin{enumerate}
			\item $V$ generates $G$ as a group.
			\item $V$ has finite stabilizer.
			\item There do not exist proper connected algebraic subgroups $H_1,H_2$ of $G$ such that the sum map $H_1\times H_2\rightarrow G$ is an isogeny and $V$ decomposes as a sum $W_1+W_2$ with each $W_i$ a subvariety of $H_i$.
		\end{enumerate}
		Then $V$ geometrically determines $G$.
	\end{introtheorem}

    In case $G$ is an abelian variety, the unique factorization of $G$ into simple factors allows us to slightly weaken the hypotheses of Theorem \ref{T: main}. Namely, in this case we can replace (1) and (2) above by the assumption that $V$ is not a single point or all of $G$; see Corollary \ref{C: determines characterization}.
	
	When $G$ is a simple abelian variety, Theorem \ref{T: main} is essentially Zilber's theorem, powered by the full Restricted Trichotomy instead of Rabinovich's Theorem. Roughly speaking, standard model-theoretic arguments (combined with the Restricted Trichotomy) provide an algebraic curve generating $G$; then one can invoke Zilber's argument. 
	
	In the general (non-simple) case, one might not be able to directly reconstruct a curve generating $G$. Instead, our proof of Theorem \ref{T: main} operates in a more general framework, where we allow `decompositions' of the pair $(G,V)$ into finitely many pairs as in Theorem \ref{T: main}. In the end, this leads to a stronger conclusion and our second main result: every instance of (1) and (2) above decomposes canonically into a sum of instances of (3):
	
	\begin{introtheorem}\label{T: general}
		Let $K$ be an algebraically closed field. Let $G$ be a semiabelian variety over $K$. Let $V$ be a closed, irreducible subvariety of $G$. Assume that $V$ generates $G$ and has finite stabilizer. Then there are:
		\begin{itemize}
			\item connected algebraic subgroups $H_1,...,H_n$ of $G$ of positive dimension, and
			\item for each $i$, a closed irreducible subvariety $W_i$ of $H_i$,
		\end{itemize}
		
		such that:
		\begin{enumerate}
			\item The sum map $H_1\times...\times H_n\rightarrow G$ is an isogeny.
			\item $V=W_1+...+W_n$.
			\item Each $W_i$ geometrically determines $H_i$.
		\end{enumerate}
		
		Moreover, the $H_1,...,H_n$ are unique, and the $W_1,...,W_n$ are unique up to translation.
	\end{introtheorem}

    So Theorem \ref{T: main} now corresponds to Theorem \ref{T: general} in the case $n=1$.
	
	
	\subsection{Relation to Mordell-Lang} The decomposition of $G$ provided by Theorem \ref{T: general} is reminiscent of, and in fact closely related to, Hrushovski's proof of the Mordell-Lang Conjecture for function fields \cite{HrMordellLang}. We do not elaborate on the statement of this conjecture; instead, we informally recall that one is given a certain expansion of a group (coming from a semiabelian variety) with unusual properties, and one needs to produce an algebraically closed field `explaining' these properties.  Let us now recall the two overarching steps in Hrushovski's proof:
    
    \begin{enumerate}
        \item[A.] First (Proposition 5.6., \emph{loc. cit.}), he proves the result if the group is algebraic over the realizations of a single minimal type (see the appendix for more on these terms; Hrushovski calls such a subgroup \textit{semi-minimal}). This case amounts to an instance of the Zilber Trichotomy (\cite{HrZil}).
        \item[B.] Second, he reduces the general case to the semi-minimal case by a series of model-theoretic and group-theoretic reductions. The key technical tool here is the \emph{socle theorem} (\cite[Proposition 4.3]{HrMordellLang}).
    \end{enumerate}

    Our Theorems \ref{T: main} and \ref{T: general} above follow a similar outline, working in the structure $(G(K),+,V(K))$ (i.e. the group $G$ with a predicate naming $V$). In essence, they are proved by combining Hrushovski's socle analysis with a different instance of the trichotomy more suited to this structure (namely the Restricted Trichotomy mentioned above).
	
	Let us elaborate on the socle analysis and how it pertains to our results. Suppose $G$ is an expansion of an abelian group of finite Morley rank (such as the group $G$ in either our setting or Hrushovski's). As in item (B) above, one wishes to in some sense reduce $G$ into semi-minimal pieces. To that end, one then defines the (model-theoretic) \textit{socle} of $G$ as the group generated by all of its connected semi-minimal subgroups. As it turns out, the socle of $G$ is itself a connected definable subgroup, and is in fact the sum of finitely many maximal connected semi-minimal subgroups. The main question we are interested in here -- and also the crucial question considered by Hrushovski -- is the relation between $G$ and its socle.
	
	Hrushovski's theorem states that if $G$ is `rigid', and $X\subseteq G$ is a definable set which is in some (model-theoretic) sense `irreducible' and with `finite stabilizer', then $X$ is `almost everywhere' contained in a unique coset of the socle. In particular, if such a set $X$ exists which is not almost contained in a coset of a proper subgroup, then $G$ is equal to its socle, and thus $G$ is a sum of semi-minimal subgroups. The reader should notice the similarity of this result with the conclusion of our Theorem \ref{T: general}. 
	
	In fact, it is possible (but not trivial) to incorporate direct applications of Hrushovski's socle theorem into our poofs. Instead, however, we have chosen to give a slightly longer variant. Our approach uses what we call the `full socle'. The hope is that the full socle is a more natural object to consider in our particular setting, so might clarify the argument for those outside model theory. Incidentally, the authors discovered the proof of Theorem \ref{T: general} internally (using the full socle) before realizing the relationship to Hrushovski's paper.

    \subsection{The role of the Trichotomy}
	
	In all of the results we have discussed -- Hrushovski's Mordell-Lang proof, Zilber's theorem on Jacobians, and the present paper -- there is a vital role played by Zilber's Trichotomy. In each case, one is given a group $G$ with some extra structure, and one wishes to definably endow $G$ with algebraic structure (i.e. realize $G$ as an algebraic group over an algebraically closed field $K$). In each case, the role of the Zilber Trichotomy is to produce the field $K$ (in a definable manner). Moreover, the Trichotomy then provides a definable finite correspondence between $G$ and some $K^n$, and this allows (with some care) to definably reconstruct the algebraic structure on $G$ over $K$.
    

    It is, however, worthwhile to point out a subtle difference in the use of the Trichotomy in these three theorems. Hrushovski's theorem statement (i.e. the portion of his paper relevant to our work) most resembles the \textit{hypotheses} of our Theorem \ref{T: main} paired with the \textit{conclusion} of our Theorem \ref{T: general}. That is, without assuming any indecomposability requirement (as in our clause (3) from Theorem \ref{T: main}), he produces a \textit{single} algebraically closed field controlling the whole of the given group $G$. This is because in his context, the statement of the Zilber Trichotomy is a bit stronger: not only does each semi-minimal component of $G$ produce an algebraically closed field, but any two such fields are definably isomorphic (so that in the end, the whole group is itself semi-minimal). Something similar happens in Zilber's paper on Jacobians, essentially because the Jacobian of a curve is generated by one copy of the curve.
    
    In contrast, in the current paper, the group $G$ may admit several definable isomorphism classes of definable algebraically closed fields. For example, suppose $G$ is a direct product of two Jacobians $\operatorname{Jac}(C_1)\times\operatorname{Jac}(C_2)$, endowed with the subvariety $C_1\times C_2$ (viewed through Abel-Jacobi embeddigs of each $C_i$); then each $(\operatorname{Jac}(J_i),C_i)$ interprets a field by Zilber's theorem, but these two fields are orthogonal. In this case, the socle of $G$ will have two semi-minimal components -- namely $\operatorname{Jac}(C_1)$ and $\operatorname{Jac}(C_2)$. In fact, under the assumptions of Theorem \ref{T: general}, the components of the full socle correspond precisely to the definable isomorphism classes of algebraically closed fields definable from $G$ (and each component is a connected algebraic group over such a field). It is for this reason that the full socle seems the more natural object to consider for the problem at hand. 
	
	

    \subsection{Factorization Theorems}

	Theorem \ref{T: general} -- or more precisely the full socle theorem underlying it (Proposition \ref{P: H=G}) -- provides a natural decomposition of a semiabelian variety with respect to certain distinguished subvarieties. The last two sections of the paper study this type of factorization in more detail. Precisely, we define a \textit{semiabelian pair} as a pair $(G,[X])$, where $G$ is a semiabelian variety and $[X]$ is an equivalence class of closed irreducible subvarieties of $G$ modulo translation. We then prove general factorization theorems for semiabelian pairs. First, Theorem \ref{T: Galois Correspondence} treats the case that $X$ has finite stabilizer and generates $G$; in this case we give a model-theoretic classification of all factorizations of the pair $(G,[X])$. As corollaries, we deduce that there is exactly one factorization into `simple' pairs; and that $(G,[X])$ is simple if and only if $X$ geometrically determines $G$.
	
	In the final section we study factorizations of pairs $(G, [X])$ for arbitrary $X$. To get a clean result in this case, we restrict to abelian varieties, and work up to \emph{weak isogeny} (Definition \ref{D: w-siog}). In this case we again characterize the simple weak isogeny classes (Lemma \ref{L: simple characterization}) and prove that every class factors uniquely into simple classes (Theorem \ref{T: unique factorization}). 

    \subsection{For those outside model theory} Our proofs are almost purely model theoretic, involving even less algebraic geometry than Zilber's proof of the Jacobian theorem  (Zilber needed a little geometry to show that the assumptions of Rabinovich's Theorem were satisfied; while in the present paper, having the full Restricted Trichotomy in our hands, this is automatic). As such, we assume familiarity with basic model theory and stability theory. \textbf{For the benefit of readers outside model theory, we have included an expository appendix in the online version.} The purpose of the appendix is to informally explain the relevant model-theoretic notions appearing in the proof, emphasizing where possible the relation to analogous notions in algebraic geometry. We hope this will allow readers to mostly follow the definitions and arguments, as well as references to more detailed literature. 
	
	\subsection{Relation to other work.}
	
		Theorem \ref{T: main} relates to work of Kollar, Lieblich, Olsson, and Sawin (\cite{KLOS}) on topological reconstruction theorems for varieties. They prove that over any uncountable field of characteristic zero, a smooth, projective variety of dimension at least two is determined by its underlying Zariski topological space (the formal statement is expressed in a similar language to Definition \ref{D: determines} above). The same statement fails for curves, and Zilber's theorem is currently the best known attempt to salvage it. Note that the theorem of \cite{KLOS} works over any field of characteristic zero, and requires a bijection preserving all Zariski closed sets. Meanwhile, our theorem requires an algebraically closed field, and for the bijection to be a group isomorphism; but we get a stronger conclusion -- i.e. the map only needs to preserve \textit{one} Zariski closed set, as long as this one set does not come from a small list of exceptions.
		
		In \cite[Conjecture 2.6]{BooVol} Booher and Voloch extend the original conjecture of Bogomolov, Korotiaev and Tschinkel (\cite{BoKoTs}) proved by Zilber to generalized Jacobians. The results of the present paper (indeed, already Zilber's Theorem augmented with the Restricted Trichotomy for algebraic curves) provide a slightly weaker version of the conjecture than the one formulated by Booher and Voloch. A proof of the full conjecture is delegated to a sequel.

\section{Full Relics}

Most of our work will be expressed in terms of \textit{relics} of algebraically closed fields (also called ACF-relics) -- this is a modern term\footnote{The terminology of relics is new, but the concepts involved are not. The notions of \textit{reducts} and \textit{interpretable structures} are very similar, but technically less suited for the current setting.}  allowing a convenient statement of the Restricted Trichotomy Conjecture mentioned above. ACF-relics were studied extensively in \cite{CasHas}. We recall the definitions:

\begin{definition}
    Let $K$ be an algebraically closed field. 
    \begin{enumerate}
        \item A \textit{$K$-relic} is a structure $\mathcal M=(M,...)$ such that every definable set in $\mathcal M$ (including $M$ itself) is interpretable in $K$.
        \item Suppose $\mathcal M=(M,...)$ is a $K$-relic, and $X$ is an interpretable set in $\mathcal M$. We say that $X$ is \textit{full in $K$ relative to $\mathcal M$} if every $K$-interpretable subset of every $X^n$ is $\mathcal M$-interpretable.
        \item A $K$-relic $\mathcal M$ is \textit{full} if its universe $M$ is full in $K$ relative to $\mathcal M$.
    \end{enumerate}
\end{definition}

So a $K$-relic is just a structure built using the interpretable sets in $K$. For example, any structure $(A(K),+,V(K))$ as in the statement of Theorem \ref{T: general} is a $K$-relic. Note that any $K$-relic has finite Morley rank -- and if $K$ is uncountable and $\mathcal M$ is a $K$-relic whose language is smaller than $|K|$, then $\mathcal M$ is saturated. 

In general, the expressive power of a $K$-relic might be significantly weaker than that of $K$. For example, the pure group $(\mathbb C,+)$ is a $\mathbb C$-relic that only recovers linear algebra over $\mathbb Q$. In contrast, \textit{full} relics reconstruct \textit{all algebraic structure}. A simple example is the complex field endowed with only addition and the square map $z\mapsto z^2$. This is a full relic, because one can define multiplication from addition and squaring: $xy$ is the unique solution to $x^2+y^2+z+z=(x+y)^2$. Note that once addition and multiplication are recovered, one automatically recovers every $(\mathbb C,+,\times)$-definable set of every $\mathbb C^n$.

Full relics capture reconstruction theorems in a general setting. That is, a proof that a $K$-relic is full is just a proof of a reconstruction theorem beginning with the data in the language of the relic. 

In \cite{CasHas}, we proved several characterizations of full ACF-relics. For the purposes of this paper, we state:

\begin{fact}\cite[Proposition 4.12]{CasHas}\label{F: full characterization} Let $K$ be an algebraically closed field, and let $\mathcal M$ be a $K$-relic. Then $\mathcal M$ is full if and only if there are an $\mathcal M$-interpretable infinite field $F$, and an $\mathcal M$-interpretable injection $M\rightarrow F^n$ for some $n$.
\end{fact}

Fact \ref{F: full characterization} allows us to relate full relics to the situation of Theorem \ref{T: main}. The following is implicit in Zilber's paper (\cite{ZilJac}):

\begin{lemma}\label{L: full implies determines} Let $G$ be an algebraic group over the algebraically closed field $K$, and $V$ a closed, irreducible subvariety of $G$. Let $\mathcal G=(G(K),\cdot,V(K))$, a $K$-relic in the language of groups with a distinguished subset. If $\mathcal G$ is a full $K$-relic, then $V$ geometrically determines $G$.
\end{lemma}
\begin{proof}
The argument is essentially that of  \cite[\S 5]{ZilJac}, so we only give a quick sketch. The assumptions provide us with the following diagram (the squiggly arrows should be read as ``the structure at the tail of the arrow interprets the one at its head"): 

\begin{center}
\begin{tikzcd}
K_1 \arrow[r, squiggly]  & G_1 \ar[d, leftrightarrow, "\phi"] \ar[r, squiggly] & F_1 \ar[r, squiggly] \ar[d, leftrightarrow, "\tilde \phi"] \ar[ll, leftrightarrow, "\simeq" description, bend right]  & H_1 \ar[ll, leftrightarrow, "\simeq" description, bend right]  \ar[d, leftrightarrow, "\hat \phi"]\\
K_2 \arrow[r, squiggly]  & G_2 \ar[r, squiggly] & F_2 \ar[r, squiggly] \ar[ll, leftrightarrow, "\simeq" description, bend left]  & H_2 \ar[ll, leftrightarrow, "\simeq" description, bend left] 
\end{tikzcd}
\end{center}
Indeed, suppose we are given isomorphic full relics $\mathcal G_i$ of $K_i$ as above. Then by Fact \ref{F: full characterization} (i.e. the Restricted Trichotomy), each $\mathcal G_i$ is \textit{bi-interpretable} with $K_i$. Precisely:

\begin{enumerate}
    \item $\mathcal G_i$ interprets an algebraically closed field $F_i$, which is $K_i$-definably isomorphic to $K_i$ (by e.g. \cite[Theorem 4.15]{PoiGroups}).
   \item Each $F_i$ then admits a relic $\mathcal H_i$ (in the signature of $\mathcal G_i$) which is $\mathcal G_i$-definably isomorphic to $\mathcal G_i$ (this is a restatement of Fact \ref{F: full characterization}).
\end{enumerate}


Therefore, an isomorphism $\phi: G_1\to G_2$  induces an isomorphism $\tilde \phi: F_1\to F_2$, which -- in turn -- induces an isomorphism $\hat \phi: H_1\to H_2$. We thus obtain an abstract field isomorphism of $K_1$ with $K_2$, by factoring as $K_1\rightarrow F_1\rightarrow F_2\rightarrow K_2$. Now if we first apply this automorphism, we may identity $K_1$ and $K_2$, and from now on write $K_1=K_2=K$ (precisely, in the language of Definition \ref{D: determines}, we are building the isomorphism $K_1\rightarrow F_1\rightarrow F_2\rightarrow K_2$ into the desired map $\sigma$).

We thus obtain the simplified diagram: 

\begin{center}
\begin{tikzcd}
K \arrow[r, squiggly] \arrow[rd, squiggly] & G_1 \ar[d, leftrightarrow, "\phi*"] \ar[r, squiggly] & F_1 \ar[r, squiggly] \ar[d, leftrightarrow, "\tilde \phi*"] \ar[ll, leftrightarrow, "\simeq" description, bend right]  & H_1 \ar[ll, leftrightarrow, "\simeq" description, bend right]  \ar[d, leftrightarrow, "\hat \phi*"]\\
 & G_2 \ar[r, squiggly] & F_2 \ar[r, squiggly] \ar[llu, leftrightarrow, "\simeq" description, bend left=80]  & H_2 \ar[ll, leftrightarrow, "\simeq" description, bend left] 
\end{tikzcd}
\end{center}

where now $K\rightarrow F_1\rightarrow F_2\rightarrow K$ is the identity on $K$. Then it follows that the resulting isomorphism $\tilde \phi^*:F_1\rightarrow F_2$ is $K$-definable (factoring it as $F_1\rightarrow K\rightarrow F_2$), and therefore so is $\hat \phi^*:H_1\rightarrow H_2$. But recall that the isomorphisms $G_i\to H_i$ are $\mathcal G_i$-definable, and thus also $K$-definable. Thus $\phi^*:G_1\rightarrow G_2$ is also $K$-definable, by factoring it as $G_1\rightarrow H_1\rightarrow H_2\rightarrow G_2$. 

 Now the definability of $\phi^*$ (and quantifier elimination in algebraically closed fields) implies that $\phi^*$ decomposes as a composition of a Frobenius automorphism and a bijective isogeny:
$G_1\xrightarrow{\mathrm{Fr}^m}  G_3\xrightarrow{I} G_2$. Finally, absorbing the Frobenius automorphism into our original field isomorphism $K_1\rightarrow K_2$ ( that we used to pass from the first diagram to the second), we obtain a factorization of the desired form and thus conclude the lemma.
\end{proof}

By Lemma \ref{L: full implies determines}, we can restate the goal of the whole paper as follows:

\begin{question}\label{Q: relic version of main question}
    Let $G$ be a semiabelian variety over the algebraically closed field $K$, and let $V$ be a closed, irreducible subvariety. For which pairs $(G,V)$ is the $K$-relic $(G(K),+,V(K))$ full?
\end{question}

Our theorem will show that $(G(K),+,V(K))$ is full precisely when the conditions of Theorem \ref{T: main} are met. Thus, for the rest of the paper, we will work mainly in the language of Question \ref{Q: relic version of main question}.

Note that the equivalent condition in Fact \ref{F: full characterization} (and its negation) are invariant under moving to a bigger model (for the model theorists, we are using here that $K$ eliminates $\exists^{\infty}$ in all sorts, and thus so does $\mathcal M$). This says that, when determining whether a $K$-relic $\mathcal M$ is full, one can assume that $K$ is arbitrarily large, and thus that both $K$ and $\mathcal M$ are saturated. Thus, for the rest of the paper, we will always assume saturation. 

We will need a much stronger version of Fact \ref{F: full characterization}, invoking Zilber's Trichotomy (or more precisely the recent solution of the \textit{Restricted Trichotomy Conjecture}, \cite{CaHaYe}). This is a general reconstruction theorem for ACF-relics. A particularly simple characterization can be given for commutative divisible groups (which include all semiabelian varieties), so we will state only this version. We also give a quick proof sketch for interested model theorists.

\begin{fact}\label{F: type characterization of full} Let $K$ be a saturated\footnote{For those unfamiliar with saturation: if the language of the relic is countable, then one could replace both instances of `saturated' in this statement with `uncountable'.} algebraically closed field, and let $\mathcal G=(G,+,...)$ be a saturated $K$-relic, where $(G,+)$ is a divisible abelian group. Then $\mathcal G$ is full if and only if its universe $G$ is almost $\mathcal G$-internal to a non-locally modular minimal type.
\end{fact}
\begin{proof}
    If $\mathcal G$ is full, then by Fact \ref{F: full characterization}, there are a $\mathcal G$-interpretable algebraically closed $F$ and a $\mathcal G$-interpretable injection $f:G\rightarrow F^n$ for some $n$. This shows that $G$ is internal to $F$. Now let $p$ be the generic type of $F$. So $p$ is minimal and non-locally modular. Note that every element of $F$ is a sum of two realizations of $p$ -- thus $F$ is internal to $p$, and thus by transitivity, so is $G$.

    Now assume $G$ is almost internal to the non-locally modular minimal type $p$. By the Zilber Trichotomy for ACF-relics (i.e. the Restricted Trichotomy Theorem, \cite{CaHaYe}), there is a $\mathcal G$-interpretable algebraically closed field internal to $p$. Thus, $G$ is almost internal to $F$ as well.
    
    Now a theorem of Hrushovski \cite[Proposition 2.25]{HrRid} shows that almost internality of a group can always be explained by a finite subgroup: that is, there is a finite subgroup $H\leq G$ such that $G/H$ is internal to $F$ (the proof of this result essentially appears in \cite[Lemma 4.3]{ZilJac}). Let $m=|H|$. Then the map $x\mapsto mx$ is well-defined on cosets of $H$, and by divisibility it defines a surjection $G/H\rightarrow G$. So $G$ is internal to $G/H$, and by transitivity, $G$ is internal to $F$. Finally, combining internality with elimination of imaginaries in $F$ gives an interpretable injection $G\rightarrow F^n$, so that Fact \ref{F: full characterization} applies.
\end{proof}

By Fact \ref{F: type characterization of full}, Theorem \ref{T: main} now amounts to determining when a divisible abelian group relic is almost internal to a non-locally modular minimal type. This is essentially the content of Hrushovski's socle theorem discussed in the introduction, and will also be our main focus moving forward. The proof will exploit the calculus of internality, orthogonality, and stable embededness (see the appendix for a discussion of these notions).

 \section{Groups Generated by Minimal Types}

    Suppose $A$ is an abelian variety over an algebraically closed field $K$, and $X$ is an irreducible closed curve in $A$. Then after a translation, $X$ generates an abelian subvariety of $A$ in finitely many steps. One can show this dimension-theoretically: the subvarieties $X$, $X+X$, $X+X+X$, ... have dimensions between 0 and $\dim(A)$, and so they must stabilize. Once they do, the resulting sum is easily seen to be a coset of the appropriate abelian subvariety.
    
    The same result and proof work in the model-theoretic setting, and we will need this. \textbf{Fix a sufficiently saturated expansion of an abelian group of finite Morley rank, say $\mathcal G=(G,+,...)$.} (Equivalently, we work in a structure as in the appendix, assuming that a commutative group structure is part of its basic language). We will show that up to translation, every minimal type in $G$ generates a definable subgroup of $G$. This is surely well-known to model theorists (see, e.g, Lascar's overview \cite[\S 7]{LasML} in \cite{BousML}), so we only give a quick proof sketch.
    
    \begin{lemma}\label{L: group generated by type} Let $p$ be a minimal type of an element of $G$. Then there is a unique subgroup $H\leq G$ such that:
    \begin{enumerate}
        \item $H$ is definable and connected.
        \item Every element of $H$ is a $\mathbb Z$-linear combination of realizations of $p$.
        \item The realizations of $p$ are contained in a single coset of $H$. 
    \end{enumerate}
    \end{lemma}
    \begin{proof}
    Without loss of generality, $p$ is a type over $\emptyset$ (otherwise, add constants to the language). Uniqueness is a straightforward and omitted (as we will only need existence). 

    For existence, let $a_1,a_2,...$ be a sequence of independent realizations of $p$. For each $n$, let $q_n=\tp(a_1+...+a_n)$. Each $q_n$ is stationary, and the ranks $U(q_n)$ are non-decreasing and uniformly bounded by the rank of $G$. Thus there are $n$ and $d$ such that $U(q_m)=d$ for all $m\geq n$.

    It follows by rank considerations that for $c\models q_n$, the translation of $q_n$ by $c$ agrees generically with $q_{2n}$. In particular, any two realizations of $q_n$ generically induce the same translation of $q_n$; and thus their difference belongs to $\operatorname{Stab}(q_n)$. It follows that $U(\operatorname{Stab}(q_n))$ is at least as large as $U(q_n)$; and this easily implies that $U(q_n)$ is the generic type of a coset of $\operatorname{Stab}(q_n)$ (in particular, we get that $\operatorname{Stab}(q_n)$ is connected). One then easily checks that $H=\operatorname{Stab}(q_n)$ satisfies the requirements of the lemma. In particular, note that every element of $H$ is the difference of two realizations of $q_n$, and thus is a linear combination of $2n$ realizations of $p$.
            \end{proof}

            \begin{definition}\label{D: G(p)}
                Let $p$ be a minimal type of an element of $G$. We call the group constructed in Lemma \ref{L: group generated by type} the \textit{connected group coset-generated by $p$}, and denote it $G(p)$.
            \end{definition}

            Recall that Zilber's theorem on Jacobians (\cite{ZilJac}) only works for curves of genus at least 2 -- essentially because in genus 1, one only recovers an elliptic curve with its pure group structure. A similar phenomenon holds in the abstract setting. Namely, one might think of Lemma \ref{L: stab vs nlm} below as an abstract analog of Zilber's theorem:

            \begin{lemma}\label{L: stab vs nlm} Let $a\in G$ be an element such that $p:=\tp(a/A)$ is minimal for some parameter set $A$. Then the following are equivalent:
            \begin{enumerate}
                \item $U(G(p))=1$.
                 \item $p$ is the generic type of a coset of a connected subgroup of $G$.
                \item $\operatorname{Stab}(p)$ is infinite.
               
            \end{enumerate}
            Moreover, if the equivalent conditions (1)-(3) fail, then $p$ is not locally modular.
            \end{lemma}
            \begin{proof} Again, this should be well-known to model theorists, so we give a sketch. The proof relies on another well-known model-theoretic fact about stable groups (the classification of definable sets in 1-based groups, \cite{HrPi87}).

                (1)$\rightarrow$(2) By construction, $p$ is contained in a coset of $G(p)$. By assumption, $G(p)$ is connected of U-rank 1. This is only possible if $p$ is the generic type of the coset containing it.
                
                (2)$\rightarrow$(3) If $p$ is the generic type of a coset of $H$, then $U(H)=U(p)=1$, so $H$ is infinite; and moreover $H$ stabilizes each of its cosets, so in particular stabilizes $p$.
                
                (3)$\rightarrow$(1) If $p$ is minimal and has infinite stabilizer, the only possibility is that $p$ is the generic type of a coset of its stabilizer (so in particular $\operatorname{Stab}(p)$ is connected of U-rank 1). But then $\operatorname{Stab}(p)$ is precisely $G(p)$, as is easy to check. So $U(G(p))=U(\operatorname{Stab}(p)=1$. 

                For the `moreover clause', we show the contrapositive. So, assume $p$ is locally modular. Since $G(p)$ is generated by $p$, it is internal to $p$, and thus every minimal type internal to $G(p)$ is locally modular. It now follows from \cite{HrPi87} that $G(p)$ has only \textit{affine structure}: every stationary type in $G(p)$ is the generic type of a coset of a connected subgroup. 
                
                In particular, this holds of some translate of $p$ into $G(p)$, and thus (translating back) the same holds of $p$. So we have verified condition (2), and (1) and (3) follow.
            \end{proof}

\section{The Setting}\label{S: ACP}

We now move toward the proofs of Theorems \ref{T: main} and \ref{T: general}.

\begin{notation} From now on, whenever $V$ is a variety over the algebraically closed field $K$, we abuse notation by abbreviating the set $V(K)$ as simply $V$.
\end{notation}

\begin{assumption}\label{A: setting} \textbf{Throughout Sections 6-9 (until the proof of Theorem \ref{T: main}), fix the following:
\begin{enumerate}
    \item An uncountable algebraically closed field $K$.
    \item A semiabelian variety $G$ over $K$. 
    \item A $K$-relic $\mathcal G=(G,+,...)$ expanding the group structure of $G$. We assume the language of $\mathcal M$ is countable.
    \item A structure $\mathcal K=(K,+,\cdot,...)$, expanding the field structure on $K$ by countably many constants capable of defining $\mathcal G$. That is, we assume that $G$, its group operation, and all basic relations of $\mathcal G$, are $\emptyset$-definable in $\mathcal G$.
    \end{enumerate}}
\end{assumption}

\begin{remark} As we go, we may add countably many additional constant symbols to the languages of $\mathcal G$ and $\mathcal K$.
\end{remark}

    Assumption \ref{A: setting}(4) gives, for example, that two tuples with the same $\mathcal K$-type also have the same $\mathcal G$-type. Moreover, if $a$ is algebraic over $b$ in the sense of $\mathcal G$, then $a$ is also algebraic over $b$ in the sense of $\mathcal K$.

    As relics of $K$, $\mathcal K$ and $\mathcal G$ have finite Morley rank. Moreover, since $K$ is uncountable and the languages of $\mathcal G$ and $\mathcal K$ are countable, it follows that $\mathcal K$ and $\mathcal G$ are saturated. Thus, we can use the machinery developed in the previous sections.

    By default, our model-theoretic analysis will take place in $\mathcal G$. However, we will occasionally use $\mathcal K$ as well, particularly when discussing dimension and genericity. Thus, we clarify:

    \begin{notation} Throughout, we use the prefixes $\mathcal G$ and $\mathcal K$ to distinguish model-theoretic terms in the two structures (e.g. $\mathcal G$-definable, $\mathcal K$-definable, $\mathcal G$-algebraic, $\mathcal K$-algebraic, etc.). If we wish to distinguish the parameters a set is defined over, we use notation such as $\mathcal G(A)$-definable. We use $\dim$ for dimension in the sense of $\mathcal K$ (that is, in the sense of varieties over $K$), and $U$ for the model-theoretic analog in $\mathcal G$. 
    \end{notation}

    In general, there may be no relationship between genericity in the senses of $\mathcal K$ and $\mathcal G$. However, there is an implication in the case of groups, which we will use repeatedly:

    \begin{lemma}\label{L: independence transfer} Let $N\leq G$ be a $\mathcal G(A)$-definable subgroup, and let $g\in N$. If $g$ is $\mathcal K$-generic in $N$ over $A$, then $g$ is $\mathcal G$-generic in $N$ over $A$.
    \end{lemma}
    \begin{proof}
        This follows because generic types in stable groups admit a purely group-theoretic characterization: a definable set $Y\subset N$ is generic in $N$ (i.e. extends to a generic type) if and only if finitely many translates of $Y$ cover $N$ (see \cite[Lemma 2.5]{PoiGroups}). Now if $g$ is not $\mathcal M$-generic over $A$, there is a $\mathcal G(A)$-definable $Y\subset N$ containing $g$ such that no finite set of translates of $Y$ contain $G$. All of this holds also in the structure $\mathcal K$, so $g$ is not $\mathcal K$-generic in $N$ over $A$ either.
    \end{proof}

\section{The Full Socle}

Our general goal is to show that, if $\mathcal G=(G,+,X)$ for $X$ as in Theorem \ref{T: general}, then $G$ is generated by full $\mathcal G$-definable subgroups; and that if the conditions of Theorem \ref{T: main} are met, there is only one such subgroup. Our main tool is the \textit{full socle} that will be defined below. 

The following facts are well-known:

    \begin{fact}\label{L: definable subgroups are abelian subvarieties} Let $H\leq G$ be a $\mathcal G$-definable subgroup. 
    \begin{enumerate}
        \item $H$ is Zariski closed, so is an algebraic subgroup.
        \item $H$ is connected in the sense of model theory if and only if it is connected as an algebraic group.
    \end{enumerate}
    \end{fact}

    Fact \ref{L: definable subgroups are abelian subvarieties}(1) is \cite[Lemma 7.4.9]{MaBook},
    while (2) follows since in both cases (algebraic and model-theoretic), connectedness is equivalent to divisibility\footnote{The equivalence of connectedness and divisibility is specific to subgroups of semi-abelian varieties, since all connected algebraic subgroups are divisible; in general there are connected groups of finite Morley rank that are not divisible.}

    By Fact \ref{L: definable subgroups are abelian subvarieties}, we will use the word \textit{connected} unambiguously when referring to definable subgroups of $G$.

    \begin{definition}
        A \textit{full subgroup of $G$} is a $\mathcal G$-definable subgroup $H\leq G$ which is full in $K$ with respect to $\mathcal G$ (that is, every constructible subset of every $H^n$ is also $\mathcal G$-definable). A \textit{maximal connected full subgroup} is a connected full subgroup not contained in any larger connected full subgroup.
    \end{definition}
    
    Note that every connected full subgroup $H$ is contained in a maximal connected full subgroup -- of all connected full subgroups of $G$ containing $H$, choose one of maximal dimension. (In fact, this subgroup turns out to be unique -- provided $H$ is non-trivial -- as a consequence of Lemma \ref{L: H_i orthogonal}(2).)

    The following basic properties are well known and easy to check using standard model theoretic techniques (see, e.g.,  \cite[\S 7]{LasML} for the proof in a closely related setting): 

    \begin{lemma}\label{L: H_i orthogonal}
        Let $H_1,...,H_n$ be distinct maximal connected full subgroups, and $H=H_1+...+H_n$. 
        \begin{enumerate}
            \item If $p$ is a minimal type almost internal to $H$, then there is exactly one $i$ such that $p$ is almost internal to $H_i$. 
            \item The $H_i$ are pairwise orthogonal.
            \item The sum map $H_1\times...\times H_n\rightarrow H$ is an isogeny.
            \item $\dim(H)=\sum_{i=1}^n\dim(H_i)$. 
        \end{enumerate}
    \end{lemma}

    In particular, Lemma \ref{L: H_i orthogonal}(4) implies that there are only finitely many maximal connected full subgroups of positive dimension.

    \begin{definition}\label{D: socle}
        Let $H_1,...,H_n$ be the distinct maximal connected full subgroups of $G$ of positive dimension. We call $H_1,...,H_n$ the \textit{full components of $\mathcal G$}. We call their sum $H=H_1+...+H_n$ the \textit{full socle of $\mathcal G$}. 
    \end{definition}

    \begin{notation}
        From now until the proof of Theorem \ref{T: main}, we fix $H_1,...,H_n$ and $H$ as in Definition \ref{D: socle}. (Note that at this stage, we do allow the case $n=0$; however, part of the content of Theorem \ref{T: general} will eventually be that $n\geq 1$).
    \end{notation}

    So by Lemma \ref{L: H_i orthogonal}, the $H_i$ are pairwise orthogonal and the sum map $H_1\times...\times H_n\rightarrow H$ is an isogeny. Note that each $H_i$ is $\acl(\emptyset)$-definable in $\mathcal G$, and $H$ is $\emptyset$-definable in $\mathcal G$ (this follows from automorphism invariance, but could also be assumed after naming finitely many constants). Note also that as a sum of connected groups, $H$ is connected. It is, in other words, the group generated by all maximal connected full subgroups. \textbf{To prove Theorems \ref{T: main} and \ref{T: general}, our main task is to show, under appropriate assumptions, that $H=G$, i.e. that $G$ is generated by its connected full subgroups.}

\section{The Algebraic Choice Property}

We now develop a key property of the full socle $H$. Roughly speaking, it says that to any non-empty definable subset of $H$, we can definably associate a canonical finite subset. This property gives a more elementary alternative (available only in our context) for some of the delicate analysis in Section 4 of Hrushovski's proof of the function-field Mordell-Lang, \cite{HrMordellLang}.

\begin{definition}
    Let $X$ be a $\mathcal G$-interpretable set.
    \begin{enumerate}
        \item Let $A\subset G^{eq}$. We say that $X$ has the \textit{algebraic choice property} (ACP) \textit{over $A$} if $X$ is $\mathcal G(A)$-interpretable, and whenever $A\subset B\subset\mathcal G^{eq}$ and $Y\subset X$ is $\mathcal G(B)$-interpretable, there is $y\in Y\cap\acl_{\mathcal G}(B)$.
        \item We say that $X$ \textit{has the ACP} if it has the ACP over some countable set $A$.
    \end{enumerate}
\end{definition}

The following are easy to check:

\begin{lemma}\label{L: ACP preservation}
    \begin{enumerate}
        \item Every $\mathcal G$-interpretable algebraically closed field has the ACP.
        \item If $X$ and $Y$ have the ACP, then so does $X\times Y$.
        \item If $f:X\rightarrow Y$ is a surjective $\mathcal G$-interpretable function, and $X$ has the ACP, then so does $Y$.
        \item If $f:X\rightarrow Y$ is a $\mathcal G$-interpretable function with finite fibers, and $Y$ has the ACP, then so does $X$. In particular, if $Y$ has the ACP, then so does every interpretable subset of $Y$.
    \end{enumerate}
    \end{lemma}

    \begin{corollary} $H$ has the ACP.
    \end{corollary}
    \begin{proof} By Fact \ref{F: full characterization}, for each $i$ there is an interpretable embedding from $H_i$ into an affine space over an interpretable algebraically closed field. By Lemma \ref{L: ACP preservation}(1), (2), and (4), each $H_i$ has the ACP. By (2), $H_1\times...\times H_n$ has the ACP. Now apply (3) to the sum map $H_1\times...\times H_n\rightarrow H$, $H$ has the ACP.
    \end{proof}

    \begin{assumption}
        For the next 3 sections, we add a set of parameters witnessing the ACP for $H$. That is, we assume moving forward that $H$ has the ACP over $\emptyset$.
    \end{assumption}
    
    \section{Strong Rigidity}
    In model theory, a stable group is called \textit{strongly rigid} if every one of its definable subgroups is $\acl(\emptyset)$-definable (this means there are no infinite definable families of subgroups). Strong rigidity played a key role in Hrushovski's proof of the function field Mordell-Lang. And, indeed:

    \begin{lemma}\label{L: Rigid} $\mathcal G$ is strongly rigid.
    \end{lemma}
    \begin{proof} If not, there is an infinite $\mathcal G$-definable (thus constructible) family of distinct subgroups. But such a family cannot exist, since the full algebraic structure on $G$ (and on any semiabelian variety) is strongly rigid (see \cite[Lemma 4.10]{HrMordellLang}). 
    \end{proof}

    Strong rigidity is used to relate the stabilizer of a subvariety $X\subset G$ to the stabilizers of smaller pieces of $X$. The idea is that if we could cover $X$ by a family of smaller sets $X_t$ with large stabilizers, then the stabilizers $\operatorname{Stab}(X_t)$ cannot vary in a family, so generically many of them are the same; and this common infinite stabilizer will then stabilize all of $X$.
    
    The next two lemmas make the above paragraph precise:

    \begin{lemma}\label{L: stabilizer of extension lemma} Let $X$ be a $\mathcal G(\emptyset)$-definable, closed, irreducible subvariety of $G$ with finite stabilizer. Let $N$ be a $\mathcal G(A)$-definable subgroup of $G$. Let $a$ be $\mathcal K$-generic in $X$ over $\emptyset$, and let $g$ be $\mathcal K$-generic in $N$ over $Aa$. If $g+a\in X$, then $N$ is finite.
    \end{lemma}
    \begin{proof}
        By rigidity, $N$ is $\operatorname{acl}(\emptyset)$-definable (in $\mathcal G$, and thus also in $\mathcal K$). So $g$ is $\mathcal K$-generic in $N$ over both $Aa$ and $\acl(\emptyset)$, which implies that $a$ and $g$ are $\mathcal K$-independent over $\emptyset$. Thus $a$ is still $\mathcal K$-generic in $X$ over $g$ -- and since $g+a\in X$, it follows that $g+X$ generically coincides with $X$. Since $X$ and $g+X$ are closed and irreducible, we conclude that they are equal. That is, $g\in\operatorname{Stab}(X)$. By assumption $\operatorname{Stab}(X)$ is finite, so that $g\in\acl(\emptyset)$. But $g$ is $\mathcal K$-generic in $N$, so it follows that $\dim(N)=0$, i.e. $N$ is finite.
    \end{proof}

    We conclude:
    
    \begin{lemma}\label{L: stab of extension is finite} Let $X$ be a $\mathcal G(\emptyset)$-definable, closed, irreducible subvariety of $G$ with finite stabilizer. Let $a\in X$ be $\mathcal K$-generic over $\emptyset$.
    \begin{enumerate}
        \item If $Y\subset G$ is $\mathcal G$-definable and contains $a$, then the set-wise stabilizer $\operatorname{Stab}(Y)$ is finite.
        \item For any $\mathcal G$-algebraically closed set $B$, the stationary type $\tp_{\mathcal G}(a/B)$ has finite (model-theoretic) stabilizer.
    \end{enumerate}
    \end{lemma}
    \begin{proof}
        \begin{enumerate}
            \item Let $N=\operatorname{Stab}(Y)$. Fix $A\subset G$ so that $N$ is $\mathcal G(A)$-definable. Let $g\in N$ be $\mathcal K$-generic over $Aa$. By definition of $N$, $g+a\in Y\subset X$, so that Lemma \ref{L: stabilizer of extension lemma} applies.
            \item Let $p=\tp_{\mathcal G}(a/B)$, and let $N=\operatorname{Stab}(p)$. So $N$ is $\mathcal G(B)$-definable (e.g., \cite[Theorem 7.1.10]{MaBook}). Let $g\in N^0$ (the connected component of $N$) be $\mathcal K$-generic over $Ba$. By Lemma \ref{L: independence transfer}, $g$ is also $\mathcal G$-generic in $N$ over $Ba$. It follows that $g$ and $Ba$ are $\mathcal G$-independent, so that $a\models p_g$, the non-forking extension of $p$ over $Bg$. So by definition of $N$, $a+g\models p$, and thus $a+g\in X$. Now apply Lemma \ref{L: stabilizer of extension lemma}. 
        \end{enumerate}
    \end{proof}

        \section{The Full Socle Theorem}

        In this section we prove an analog of Hrushovski's socle theorem (\cite[Proposition 4.3]{HrMordellLang}). One could also deduce our result from Hrushovski's via a short (but not completely trivial) argument. We include a sketch of a direct proof for readers unfamiliar with Hrushovski's version (and because we can give a slightly simpler proof in our setting).

        Recall that we are still working under Assumption \ref{A: setting}, and $H$ denotes the full socle of $G$, with full components $H_1,...,H_n$.

        \begin{proposition}\label{P: H=G} Let $X$ be a $\mathcal G(\emptyset)$-definable, closed, irreducible subvareity of $G$ with finite stabilizer. Then $X$ is contained in a single coset of $H$.
        \end{proposition}
        \begin{proof} \textbf{Throughout the proof, fix $a_0$, a $\mathcal K$-generic element of $X$. Let $C_0$ be the coset $H+a_0$. When referring to $C_0$ as an element (of the interpretable group $G/H$) rather than a set, we write the lowercase $c_0$.}

        First suppose $U(c_0)=0$. Then $a_0$ is still generic in $X$ over $c_0$, which implies that a dense open subset of $X$ belongs to $C_0$. Since $C_0$ is closed, this means $X$ is contained in $C_0$, which gives the desired statement.

        Now assume $U(c_0)\geq 1$. We seek a contradiction. By definition of $U$-rank we can fix a $\mathcal G$-algebraically closed set $B\subset G$ so that $U(c_0/B)=1$. We first use the ACP and stable embededness to `identify' $c_0\in G/H$ with an element of $G$:

        \begin{claim} There is $b_0\in X\cap C_0$ $\mathcal G$-interalgebraic with $c_0$ over $\emptyset$.
        \end{claim}
        \begin{proof} Consider the $\mathcal G(c_0)$-definable family of translates $Y_b=(X\cap C_0)-b$ for $b\in X\cap C_0$. Each $Y_b$ is a subset of $H$. By stable embededness, there is a $\mathcal G(c_0)$-definable reparametrization of this family by $H$-tuples, say $\{Z_t:t\in T\}$ with $T\subset H^n$. By the ACP, there is $t_0\in T\cap\acl_{\mathcal G}(c_0)$. Then by definition, there is $b\in X\cap C_0$ with $(X\cap C_0)-b=Y_{t_0}$. By Lemma \ref{L: stab of extension is finite}, $X\cap C_0$ has finite stabilizer, so there are only finitely many such $b$, and thus each one is $\mathcal G$-algebraic over $c_0t_0$ (and thus over $c_0$). Let $b_0$ be any such $b$. So $b_0\in\acl_{\mathcal G}(c_0)$, and also $c_0\in\acl_{\mathcal G}(b_0)$ since $C_0$ is the coset containing $b_0$.
        \end{proof}

        Fix $b_0$ as in the claim, and let $p=\tp_{\mathcal G}(b_0/B)$. By the claim, $U(b_0/B)=U(c_0/B)=1$. So since $B$ is algebraically closed, $p$ is a minimal type. We will work with the subgroup $G(p)$ of Definition \ref{D: G(p)}.

        Since $U(c_0/B)=1$, the group $G(p)$ is not contained in $H$ (in fact, it intersects infinitely many cosets of $H$). Thus $G(p)$ cannot be full. By Fact \ref{F: type characterization of full}, $p$ must be locally modular. So the equivalent conditions (1)-(3) of Fact \ref{L: stab vs nlm} all hold. In particular, $\operatorname{Stab}(p)$ is infinite. We now give a more intricate version of the argument in Lemma \ref{L: stab of extension is finite} to show that $\operatorname{Stab}(p)$ is actually finite, which will give a contradiction.

        \begin{claim}
            $\operatorname{Stab}(p)$ is finite.
        \end{claim}
        \begin{proof}
            Let $N=\operatorname{Stab}(p)$, so $N$ is $\mathcal G(B)$-definable. Let $g\in N$ be $\mathcal K$-generic over $Ba_0$. Note that $b_0\in\acl_{\mathcal G}(c_0)\subset\acl_{\mathcal G}(a_0)$, so that $g$ is also $\mathcal K$-generic over $Ba_0b_0$. By Lemma \ref{L: independence transfer}, $g$ is also $\mathcal G$-generic in $N$ over $Ba_0b_0$. In particular, $g$ and $b_0$ are $\mathcal G$-independent over $B$. Then by symmetry of forking, we get that $b_0\models p_g$, the non-forking extension of $p$ over $Bg$.

            Let $b_1=g+b_0$, and let $a_1=g+a_0$. Note by definition of $g$ that $b_1\models p$. 
            
            Now let $h=a_0-b_0=a_1-b_1\in H$. By local modularity and Lemma \ref{L: H_i orthogonal}(1), $p$ must be orthogonal to $H$. So $\tp_{\mathcal G}(b_0/Bh)=\tp_{\mathcal G}(b_1/Bh)$. In particular, since $a_0=h+b_0\in X$, it follows that $a_1=h+b_1\in X$. That is, $a_0+g\in X$. The claim now follows by Lemma \ref{L: stabilizer of extension lemma}.
            \end{proof}

            Finally, we have shown that $\operatorname{Stab}(p)$ is both finite and infinite, a contradiction. Thus we rule out the case that $U(c_0)\geq 1$, which proves Proposition \ref{P: H=G}.
            \end{proof}

            We conclude that if $X\subset G$ satisfies the hypotheses of Theorem \ref{T: general}, then $G$ is generated by its full components; and moreover, the structure $\mathcal G$ is determined entirely by the full components:

            \begin{corollary}\label{C: socle is whole thing} Assume there is a $\mathcal G(\emptyset)$-definable, closed, irreducible subvariety $X\subset G$ which has finite stabilizer and generates $G$. Then:
            \begin{enumerate}
                \item $H=G$, i.e. $G$ is generated by its full components.
                \item $\mathcal G$ is precisely the structure induced via the sum map by the disjoint union of the $H_i$ (each with its full structure). In other words, for each $m$, the definable subsets of $G^m$ are precisely the finite unions of sets $Z_1+...+Z_n$ with each $Z_i$ a constructible subset of $H_i^n$.      
            \end{enumerate}
            \end{corollary}
            \begin{proof} (1) is clear by Proposition \ref{P: H=G}. We show (2). Since each $H_i$ is full, it is clear that any set of the desired form is $\mathcal G$-definable. We show the converse. Now by (1), we can write $G$ as the image of $H_1\times...\times H_n$ under the sum map. Then the definable subsets of $G^m$ are precisely the images of definable subsets of $(H_1)^m\times...\times(H_n)^m$ under addition. By orthogonality, every definable subset of $(H_1)^m\times...\times(H_n)^m$ is a finite union of products of definable subsets of the $(H_i)^m$ (for stationary sets this is an easy and well-known equivalent formulation of orthogonality, see, e,g, \cite[Definition 4.2]{HrMordellLang}). This implies the desired statement.
            \end{proof}

            \section{The Proof of Theorem \ref{T: main}}
            We have finished the main part of the proof of Theorem \ref{T: main}. We now move toward finishing the proof. Our goal now is to show that any $\mathcal G$-definable subvariety as in Theorem \ref{T: general} decomposes into a sum of subvarieties of the $H_i$. This is a consequence of orthogonality.

            We first note a general fact about isogenies of algebraic groups. We write this separately because we will use it later on as well.

            \begin{lemma}\label{L: isogeny closed} Let $S$ and $T$ be connected, commutative algebraic groups, and $f:S\rightarrow T$ an isogeny. If $V\subset S$ is closed and irreducible, then so is $f(V)$.
            \end{lemma}
            \begin{proof} This is because isogenies are closed maps (the irreducibility clause is then clear). To show $f$ is closed, factor $f$ as $S\mapsto S/\operatorname{ker(f)}\rightarrow T$. The first map is just the quotient, and the second is a universal homeomorphism. So both factor maps are closed, thus so is $f$.
            \end{proof}

            Returning to our structure $\mathcal G$, we conclude:
            
            \begin{proposition}\label{P: Z_i} Let $X\subset H$ be a $\mathcal G(\emptyset)$-definable, closed, irreducible subvariety with finite stabilizer and which is not contained in any coset of a proper algebraic subgroup of $G$. Then there are closed irreducible subvarieties $X_i\subset H_i$ for each $i$ such that $X=X_1+...+X_n$. 
            \end{proposition}
            \begin{proof} By Corollary \ref{C: socle is whole thing}, $X$ is a finite union of sets of the form $X_j=Z_1^j+...+Z_n^j$, with each $Z_i^j$ a constructible subset of $H_i$. Since $X$ is closed, we may assume all $Z_i^j$ are closed. Refining the decomposition if necessary, we may then assume all $Z_i^j$ are irreducible. Finally, since $X$ is irreducible, there is $j$ such that $X=Z_1^j+...+Z_n^j$. Now take $X_i=Z_i^j$.

            \end{proof}
            
            We are now ready to prove Theorem \ref{T: main}:

            \begin{proof}[Proof of Theorem \ref{T: main}] Let $K$, $G$, and $X$ be as in Theorem \ref{T: main}. Let $\mathcal G=(G,+,X)$, a group with a distinguished subset. By Lemma \ref{L: full implies determines}, it is enough to show that $\mathcal G$ is a full $K$-relic. As described in the introduction, we may assume by Fact \ref{F: full characterization} that $K$ is uncountable (see the discussion after Question \ref{Q: relic version of main question}).
            
            Let $H_1,...,H_n$ be the full components of $\mathcal G$, and let $H=H_1+...+H_n$ be the full socle. Note that $X$ satisfies the assumptions of Corollary \ref{C: socle is whole thing}, so we have $H=G$. Thus $X\subset H$, so by Proposition \ref{P: Z_i}, there are closed irreducible subvarieties $X_i\subset H_i$ such that $X=X_1+...+X_n$. By the hypotheses of Theorem \ref{T: main}, such a decomposition is impossible for $n>1$. Thus $n=1$, and so $G=H=H_1$ is full.
            \end{proof}

            \section{Factorizations of Subvarieties}

            \subsection{Semiabelian Pairs}

            We now drop Assumption \ref{A: setting} -- though we continue throughout to work over a fixed algebraically closed field $K$. We will be moving toward Theorem \ref{T: general}. In view of what we have already done, our main remaining goals are:
            \begin{enumerate}
                \item Prove that the subvarieties $X_i$ from Proposition \ref{P: Z_i} geometrically determine $H_i$.
                \item Prove the uniqueness of the decomposition claimed in Theorem \ref{T: general}.\end{enumerate}
                
                In fact, we will deduce (1) and (2) by proving more general unique factorization results for subvarieties of semiabelian varieties. These factorization results will only work up to translation. Thus, to streamline the presentation, we define: 

            \begin{definition}
                A \textit{semiabelian pair} is a pair $(G,[X])$, where $G$ is a semiabelian variety over $K$, and $[X]$ is an equivalence class of closed irreducible subvarieties of $G$ modulo translation. If $G$ is an abelian variety, we also call $(G,[X])$ an \textit{abelian pair}.
            \end{definition}

            \begin{example}
                A canonical example of a semiabelian pair is the example $(J,[C])$ considered by Zilber, where $C$ is a curve and $J$ is its Jacobian. Here we use standard abuse of notation to identify $C$ with its image in $J$. Typically, one can only do this after fixing a basepoint in $C$. In our case, this is not needed, because the equivalence class $[C]$ does not depend on the base point.
            \end{example}

            Clearly, if $(G,[X])$ and $(H,[Y])$ are semiabelian pairs, then $(G\times H,[X\times Y])$ is a well-defined semiabelian pair. Similarly, if $(G,[X])$ is a semiabelian pair, and $f:G\rightarrow H$ is an isogeny, then $(H,[f(X)])$ is a well-defined semiabelian pair (this uses Lemma \ref{L: isogeny closed} to get that $f(X)$ is closed).
            
            Lemma \ref{L: isogeny} says that semiabelian pairs are also closed under isogeny preimages:

            \begin{lemma}\label{L: isogeny}
                Let $f:G\rightarrow H$ be an isogeny of semiabelian varieties over $K$, and let $(H,[Y])$ be a semiabelian pair. Then there is exactly one semiabelian pair $(G,[X])$ mapping to $(H,[Y])$ under $f$.
            \end{lemma}
            \begin{proof}
                For existence, let $X$ be any top-dimensional irreducible component of $f^{-1}(Y)$. Then by closedness and irreducibility (applying Lemma \ref{L: isogeny closed}), $f(X)=Y$.

                For uniqueness, suppose $X'$ is another closed irreducible subvariety of $G$ with $f(X')=Y$. Then $f(X)=f(X')$, which implies that every element of $X'$ differs from $X$ by an element of the kernel $\operatorname{ker}(f)$. That is, $X'\subset X+\operatorname{ker}(f)$. Since $\operatorname{ker}(f)$ is finite, and $X'$ is irreducible, there is $g\in\operatorname{ker}(f)$ with $X'\subset X+g$. By closedness and irreducibility (and since $\dim(X)=\dim(X')$), this means $X'=X+g$, thus $[X']=[X]$.
            \end{proof}

            Moving forward, if $f$ is an isogeny sending the semiabelian pair $(G,[X])$ to the semiabelian pair $(H,[Y])$, we may also say that \textit{$f$ is an isogeny from $(G,[X])$ to $(H,[Y])$.}

            \subsection{Factorizations}
            We can now define factorizations of semiabelian pairs:

            \begin{definition}\label{D: decomposition}
                Let $(G,[X])$ be a semiabelian pair. 
                \begin{enumerate}
                    \item A \textit{factorization} of $(G,[X])$ is a finite sequence of semiabelian pairs $(G_i,[X_i])$, where each $G_i\leq G$, so that the sum map $s:G_1\times...\times G_n\rightarrow G$ is an isogeny from $(G_1\times...\times G_n,[X_1\times...\times X_n])$ to $(G,[X])$.
                    \item $(G,[X])$ is \textit{simple} if $\dim(G)\geq 1$ and $(G,[X])$ has no nontrivial factorizations.
                \end{enumerate}
            \end{definition}

            \begin{example}
                Suppose $(J,[C])$ is the abelian pair given by a curve $C$ and its Jacobian $J$. Then $(J,[C])$ is always simple. Indeed, if we factor $(J,[C])$ into $(J_1,[C_1])$ and $(J_2,[C_2])$, then since $C$ is irreducible of dimension 1, some $C_i$ is a single point. But then since $C$ generates $J$, it follows that $J_i$ is also a single point, so the factorization was trivial.
            \end{example}

            By Lemma \ref{L: isogeny}, a factorization $(G_1,[X_i]),...,(G_n,[X_n])$ of $(G,[X])$ is determined by the subgroups $G_i$. Indeed, this follows since the product $(G_1\times...\times G_n,[X_1\times...\times X_n])$ must be the unique preimage of $(G,[X])$ under the sum isogeny. Thus, we will often identify a factorization with its sequence of subgroups $\{G_i\}$, and say that the sequence $\{G_i\}$ \textit{determines a factorization} of $(G,[X])$.

            Note that if $\{(G_i,[X_i])\}$ is a factorization of $(G,[X])$, and for each $i$, $\{(G_i^j,[X_i^j])\}$ is a factorization of $(G_i,[X_i])$, then $\{(G_i^j,[X_i^j])\}$ (over all $i$ and $j$) is a factorization of $(G,[X])$. We call this a \textit{refinement} of the original factorization $\{(G_i,[X_i])\}$. The following is then clear:

            \begin{lemma}\label{L: existence of factorization}
            Every semiabelian pair admits a factorization into simple semiabelian pairs.
            \end{lemma}
            \begin{proof} Let $(G,[X])$ be a semiabelian pair. The number of factors in a factorization of $(G,[X])$ is bounded by the dimension of $G$. So, take a factorization $\{(G_i,[X_i])\}$ with the largest number of factors. If some $(G_i,[X_i])$ were non-simple, then factoring it would lead to a refinement of the original factorization $\{(G_i,[X_i])\}$ with more factors, a contradiction.
            \end{proof}

            \subsection{A Galois Correspondence}

            Let us reformulate our work to this point in the language we have developed. Suppose $G$ is a semiabelian variety over $K$, let $X$ be a closed irreducible subvariety which has finite stabilizer and generates $G$, and let $\mathcal G$ be any $K$-relic expanding $(G,+,X)$. Then:
            
            \begin{enumerate}
                \item Theorem \ref{T: main} says that if $(G,[X])$ is simple, then $\mathcal G$ is full.
                \item Proposition \ref{P: Z_i} says that the full components of $\mathcal G$ determine a factorization of $(G,[X])$.
            \end{enumerate}

            We now give a more general picture capturing the above two statements. Namely, suppose $(G,X)$ is as in Theorem \ref{T: general}. We give a Galois-style correspondence between factorizations of $(G,[X])$ and $K$-relics expanding $(G,+,X)$ (note that if $[X]=[Y]$, then any $K$-relic expanding $(G,+,X)$ also defines $Y$ and vice versa, so the collection of relics under consideration only depends on $(G,[X])$).
            
            Precisely, say that two such $K$-relics $\mathcal M$ and $\mathcal N$ are \textit{interdefinable} if they have the same definable sets (i.e. every $\mathcal M$-definable subset of every $M^n$ is also $\mathcal N$-definable and vice versa; after naming a small set of parameters, this is equivalent to $\mathcal M$ and $\mathcal N$ having the same automorphism group). The $K$-relics expanding $(G,+,X)$, viewed up to interdefinability, form a partial order: $\mathcal M\sqsubseteq\mathcal N$ if every $\mathcal M$-definable set is also $\mathcal N$-definable. Similarly, the factorizations of $(G,[X])$ form a partial order under refinement. In this language, we show:

            \begin{theorem}\label{T: Galois Correspondence} Let $(G,[X])$ be a semiabelian pair over $K$, where $X$ has finite stabilizer and generates $G$. Then there is an order-reversing bijective correspondence between factorizations of $(G,[X])$ and $K$-relics expanding $(G,+,X)$ up to interdefinability.
            \end{theorem}
            \begin{proof} To a $K$-relic $\mathcal G$ expanding $(G,+,X)$, we associate the factorization determined by the full components of $\mathcal G$ (i.e. the one given by Proposition \ref{P: Z_i}). Conversely, given a factorization $\{(G_i,[X_i])\}$, we associate the $K$-relic on $\mathcal G$ endowing each $G_i$ separately with its full structure. That is, $\mathcal G$ is a structure as in Corollary \ref{C: socle is whole thing} -- the definable subsets of $G^m$ are precisely the finite unions of sums of constructible subsets of the $(G_i)^m$. Note that $X$ is indeed $\mathcal G$-definable, because it has this form (as the sum of the $X_i\subset G_i$). 

            Let us first show that these two associations are inverses. Given a $K$-relic $\mathcal G$ expanding $(G,+,X)$, let $G_1,...,G_n$ be its full components, and let $\mathcal G'$ be the structure associated to the factorization determined by $\{G_i\}$. Then by Corollary \ref{C: socle is whole thing}(2), we have $\mathcal G=\mathcal G'$ (up to interdefinability).

            Now consider a factorization $\{(G_i,[X_i])\}_{i=1}^n$. Let $\mathcal G$ be the associated $K$-relic, and let $H_1,...,H_m$ be its full components. By the choice of $\mathcal G$, the groups $G_i$ are full. It follows that each $G_i$ is a subgroup of some $H_j$. Also by the choice of $\mathcal G$, the groups $G_i$ are orthogonal. Thus each $H_j$ contains at most one $G_i$ (as $H_j$ is full, any two of its definable subgroups are non-orthogonal). Rearranging if necessary, we may thus assume $n\leq m$ and each $G_i\leq H_i$. We thus have the inclusion map $f:G_1\times...\times G_n\rightarrow H_1\times...\times H_m$. Since both sides are factorizations, the domain and target of $f$ have the same dimension, i.e. $\dim(G)$. So, since $f$ is injective (and the target is connected), it is surjective. This implies that $m=n$ and each $G_i=H_i$, which is what we needed to show.

            Now it remains to show that the correspondence is order-reversing. Consider a factorization $\{(G_i,[X_i])\}$, and a refinement $\{(G_i^j,[X_i^j]\}$, where for fixed $i$, $\{(G_i^j,[X_i^j])\}$ is a factorization of $(G_i,[X_i])$. Let $\mathcal G$ be the structure associated to the coarser factorization, and $\mathcal G^{ref}$ the structure associated to the refinement. Then in $\mathcal G$, each $G_i$ is full -- thus so is each $G_i^j$ (as $G_i^j\leq G_i$). So $\mathcal G$ can define all of the basic relations of $\mathcal G^{ref}$, as desired.
            \end{proof}
            
            \subsection{Corollaries and Theorem \ref{T: general}}

            We now give some further corollaries of the Galois correspondence from Theorem \ref{T: Galois Correspondence}. First, we show that the factorization into simple pairs is unique:
            
            \begin{corollary}\label{C: uniqueness} Let $(G,[X])$ be a semiabelian pair, where $X$ has finite stabilizer and generates $G$. Then $(G,[X])$ has exactly one factorization into simple semiabelian pairs.
            \end{corollary}
            \begin{proof} A factorization into simple pairs is the same as a maximal factorization. By the Galois correspondence, maximal factorizations are identified with minimal relics expanding $(G,+,X)$. There is only one such relic -- namely $(G,+,X)$ itself.
            \end{proof}

            Next, we show that under the assumptions of Theorem \ref{T: Galois Correspondence}, fullness characterizes simplicity. This gives another way to prove Theorem \ref{T: main}:
            
            \begin{corollary}\label{C: simple iff full}
                Let $(G,[X])$ be a semiabelian pair, where $X$ has finite stabilizer and generates $G$. Then $(G,[X])$ is simple if and only if the $K$-relic $(G,+,X)$ is full.
            \end{corollary}
            \begin{proof} Note that $(G,+,X)$ is full if and only if there is exactly one $K$-relic expanding it. By the Galois correspondence, this is equivalent to the pair $(G,[X])$ having only one factorization -- which is the same as saying $(G,[X])$ is simple. 
            \end{proof}
            
            Before proving Theorem \ref{T: general}, we need to give a converse of Theorem \ref{T: main}. That is, we show that if $(G,[X])$ has a non-trivial factorization, $X$ does not geometrically determine $G$ (see Definition \ref{D: determines}). This was sketched in the introduction, but here we give the details.
            
            \begin{lemma}\label{L: determines implies simple} Let $(G,[X])$ be a semiabelian pair. If $X$ geometrically determines $G$, then $(G,[X])$ is simple.
            \end{lemma}
            \begin{proof} Suppose $(G,[X])$ is not simple. So there is a factorization into two pairs, say $(G_1,[X_1])$ and $(G_2,[X_2])$. We will build a group automorphism of $G$ which stabilizes $X$ and cannot be expressed as the composition of a field automorphism with an isogeny.

            Let $F$ be a finitely generated field capable of defining each of the following:
            \begin{itemize}
                \item $(G,+,X)$.
                \item Each $(G_i,+,X_i)$.
                \item Each element of the kernel of the sum map $G_1\times G_2\rightarrow G$.
            \end{itemize}

            If $\sigma,\tau\in\operatorname{Gal}(K/F)$, there is an automorphism $(\sigma,\tau)$ of $(G,+,X)$ which acts as $\sigma$ on $G_1$ and $\tau$ on $G_2$ (this map is well-defined because $\sigma$ and $\tau$ both act trivially on the finite intersection $G_1\cap G_2\subseteq F$).

            Now for $\tau\in\operatorname{Gal}(K/F)$, consider the automorphism $(\operatorname{id},\tau)$ as above (acting trivially on $G_1$). If $X$ geometrically determines $G$, we can write $(\operatorname{id},\tau)$ as a composition $f\circ\sigma$, where $\sigma:G\rightarrow G'$ is (induced by) an automorphism of $K$ and $f:G'\rightarrow G$ is a bijective isogeny.

            So $f\circ\sigma$ gives the identity in $G_1$, which means that $\sigma$ is uniquely determined by the pair $(G',f)$ (as $\sigma$ must invert $f$ on $G_1$-points, and we can read off the action of $\sigma$ on $K$ from the coordinates of $G_1$-points). But then, arguing similarly in $G_2$ using the equation $\tau=f\circ\sigma$, it follows that $\tau$ is also determined by the pair $(G',f)$. In particular, the number of $\tau\in\operatorname{Gal}(K/F)$ for which $(\operatorname{id},\tau)$ factors in the desired form is at most the number of isogenies over $K$ with image $G$. Each such isogeny is determined by finitely many coefficients, so the number of such isogenies is at most $|K|$. 
            
            On the other hand, the group $\operatorname{Gal}(K/F)$ has strictly greater cardinality than $|K|$. This can be seen in cases. If $K$ has infinite transcendence degree over $F$, then (as $F$ is finitely generated) the transcendence degree of $K$ over $F$ is exactly $|K|$; now pick a transcendence basis, and any permutation of the basis extends to an automorphism. On the other hand, if $K$ has finite trascendence degree over $F$, then after enlargening $F$, we may assume $K$ is the algebraic closure of $F$. Then $\operatorname{Gal}(K/F)$ is profinite and infinite, and thus uncountable (so bigger than $K$).

            Now we have shown that the size of $\operatorname{Gal}(K/F)$ is strictly greater than the number of elements $\tau$ so that $(\operatorname{id},\tau)$ factors in the desired form. So, there is $\tau$ for which $(\operatorname{id},\tau)$ does not factor in the desired form, and thus $X$ does not geometrically determine $G$.
            \end{proof}

            Combining Lemmas \ref{L: full implies determines} and \ref{L: determines implies simple} and Corollary \ref{C: simple iff full} now gives:

            \begin{corollary}\label{C: big equivalence}
                Let $(G,[X])$ be a semiabelian pair, where $X$ has finite stabilizer and generates $G$. Then the following are equivalent:
                \begin{enumerate}
                    \item $(G,+,X)$ is full.
                    \item $X$ geometrically determines $G$.
                    \item $(G,[X])$ is simple.
                \end{enumerate}
            \end{corollary}
            
            Finally, we are ready to prove Theorem \ref{T: general}:

            \begin{proof}[Proof of Theorem \ref{T: general}]
                We are given a semiabelian pair $(G,[X])$ where $X$ has finite stabilizer and generates $G$. Note that the same holds for any factorization $\{(G_i,[X_i])\}$ -- that is, each $X_i$ has finite stabilizer and generates $G_i$ (otherwise the corresponding property would fail in $(G,[X])$). Thus Corollary \ref{C: big equivalence} applies to the factors in any factorization of $(G,[X])$. In particular, a factorization of $(G,[X])$ into simple semiabelian pairs is the same as a factorization into `geometrically determined pairs' (pairs $(G_i,[X_i])$ where $X_i$ geometrically determines $G_i$). By Corollary \ref{C: uniqueness}, there is exactly one such factorization, and this is equivalent to the statement of Theorem \ref{T: general}.
            \end{proof}

        \section{More General Factorizations}

        So far, our best results on factorizations of semiabelian pairs $(G,[X])$ concern the case where $X$ has finite stabilizer and generates $X$. In this section, we prove slightly weaker results without these hypotheses. The most complete result will only work if $G$ is an abelian variety. This is because we will use the usual decomposition of an abelian variety into simple factors. Regardless, we still work in the language of semiabelian pairs, clarifying as we go when we need to be more restrictive.

        \subsection{Weak Isogeny} Recall that isogenies define an equivalence relation on semiabelian varieties. That is, if there is an isogeny from $G$ to $H$, then there is an isogeny from $H$ to $G$. This fails for isogenies of semiabelian pairs, essentially because if $f:G\rightarrow H$ sends $X\subset G$ to $Y\subset H$, the isogeny $g:H\rightarrow G$ need not send $Y$ to $X$:

        \begin{example}\label{E: isogeny counterexample} Let $C$ be a smooth irreducible projective curve of genus $g\geq 2$ over the complex numbers, such that the (algebraic) automorphism group of $C$ is trivial. Consider the abelian pair $(J,[C])$, where $J$ is the Jacobian of $C$. After distinguishing a point in $C$, view $C$ as a closed curve in $J$ containing $0$. For any $n\geq 1$,  we have an isogeny of semiabelian pairs $f:(J,[C])\rightarrow(J,[nC])$ given by scaling by $n$; but we claim that for $n\geq 2$ there is no isogeny of semiabelian pairs $g:(J,[nC])\rightarrow(J,[C])$. Indeed, let $g$ be such an isogeny. Then $g(nC)$ is a translate of $C$, say $g(nC)=C+a$. So $x\mapsto g\circ f(x)-a$ is a (necessarily finite) morphism $C\rightarrow C$. By Riemann-Hurwitz (and the fact that $g\geq 2$), this map is an automorphism, and is thus the trivial automorphism. So for $x\in C$, we have $g\circ f(x)-a=x$, i.e. $g\circ f(x)-x$ is constant. Since $0\in C$ and $f$ and $g$ are isogenies, we have $g\circ f(0)-0=0$. So $g\circ f(x)-x$ is constant 0 on $C$. Since $C$ generates $J$, $g\circ f(x)=x$ holds on all of $J$. That is, $g=f^{-1}$. But for $n\geq 2$, $J$ has non-trivial $n$-torsion, and so scaling by $n$ is not invertible. 
        \end{example}

        To remedy the situation, we define:

        \begin{definition}\label{D: w-siog}
            Let $(G_1,[X_1])$ and $(G_2,[X_2])$ be semiabelian pairs. We say that $(G_1,[X_1])$ and $(G_2,[X_2])$ are \textit{weakly isogenous} if there are a semiabelian pair $(H,[Y])$ and isogenies $(H,[Y])\rightarrow(G_1,[X_1])$ and $(H,[Y])\rightarrow(G_2,[X_2])$.
        \end{definition}

        Note that the pairs $(J,[C])$ and $(J,[nC])$ from Example \ref{E: isogeny counterexample} are weakly isogenous, witnessed by the isogenies $\operatorname{id}:(J,[C])\rightarrow(J,[C])$ $f:(J,[C])\rightarrow(J,[nC])$. And, indeed, we have:

        \begin{lemma}\label{L: iso equiv rel}
            Weak isogeny gives an equivalence relation on semiabelian pairs.
        \end{lemma}
        \begin{proof}
            Reflexivity and Symmetry are clear. For transitivity, suppose we have isogenies $(H_1,[Y_1])\rightarrow(G_1,[X_1])$, $(H_1,[Y_1])\rightarrow(G_2,[X_2])$, $(H_2,[Y_2])\rightarrow(G_2,[X_2])$, and $(H_2,[Y_2])\rightarrow(G_3,[X_3])$. We show that $(G_1,[X_1])$ and $(G_3,[X_3])$ are weakly isogenous.
            
            Indeed, let $H_3$ be the connected component of the fiber product of $H_1$ and $H_2$ over $G_2$. Then $H_3\rightarrow H_1$ and $H_3\rightarrow H_2$ are isogenies. By Lemma \ref{L: isogeny}, there is $Y_3$ so that $H_3\rightarrow H_1$ determines an isogeny of semiabelian pairs $(H_3,[Y_3])\rightarrow(H_1,[Y_1])$. Thus $H_3\rightarrow H_1\rightarrow G_1$ gives an isogeny $(H_3,[Y_3])\rightarrow(G_1,[X_1])$. We claim there is also an isogeny $(H_3,[Y_3])\rightarrow(G_3,[X_3])$. Indeed, denote the image of $(H_3,[Y_3])$ in $H_2$ by $(H_2,[Y_4])$. Since the maps $H_3\rightarrow H_1\rightarrow G_2$ and $H_3\rightarrow H_2\rightarrow G_2$ agree, we have the isogeny $(H_2,[Y_4])\rightarrow(G_2,[X_2])$. By Lemma \ref{L: isogeny}, the preimage of $[X_2]$ in $H_2$ is unique; so $[Y_4]=[Y_2]$, and thus the composition $H_3\rightarrow H_2\rightarrow G_3$ sends $[Y_3]\rightarrow[Y_2]\rightarrow[X_3]$ as desired. 
        \end{proof}

        Weak isogeny is thus the equivalence relation generated by isogenies.

        \subsection{Factorization of Weak Isogeny Classes}

        If $(G,[X])$ is a semiabelian pair, we denote by $[G,X]$ its weak isogeny class. Given two weak isogeny classes $[G,X]$ and $[H,Y]$, the product $[G\times H,X\times Y]$ is a well-defined weak isogeny class (this is easy to check). Thus we have a direct product operation on weak isogeny classes. We also have a trivial class, consisting only of the trivial group with the identity element. We now define:

        \begin{definition}
            A weak isogeny class is \textit{simple} if it is not the trivial class, and it cannot be factored as a product of two non-trivial classes.
        \end{definition}

        If $f:(G,[X])\rightarrow(H,[Y])$ is an isogeny of semiabelian pairs, then $f$ identifies factorizations of $(G,[X])$ with factorizations of $(H,[Y])$ and vice versa (via Lemmas \ref{L: isogeny closed} and \ref{L: isogeny}). Thus, factorizations of $[G,X]$ as a weak isogeny class correspond with factorizations of the semiabelian pair $(G,[X])$ up to weak isogeny. It follows, for example, that a weak isogeny class is simple if and only if some (equivalenly any) of its members is simple.

        As in Lemma \ref{L: existence of factorization}, the existence of simple factorizations is automatic:

        \begin{lemma}\label{L: weak existence of factorization} Every weak isogeny class can be factored into simple classes.
        \end{lemma}
        \begin{proof} Choose a factorization with the largest number of factors (which exists because every factorization has at most $\dim(G)$ factors). Then each factor must be simple, or else we could strictly refine the factorization.
        \end{proof}

        Moreover, our previous work gives us:

        \begin{lemma}\label{L: uniqueness} Let $[G,X]$ be a weak isogeny class, where $X$ has finite stabilizer and generates $G$. Then the factorization of $[G,X]$ into simple classes is unique.
        \end{lemma}
        \begin{proof} This is a restatement of Corollary \ref{C: uniqueness}.
        \end{proof}

        \begin{remark}\label{R: conditions isogeny invariant}
            One should note that the conditions of Lemma \ref{L: uniqueness} are weak isogeny invariant. Indeed, if $f:(G,[X])\rightarrow(H,[Y])$ is an isogeny of semiabelian pairs, once checks easily that some (equivalently any) representative of $[X]$ generates $G$ if and only some (equivalently any) representative of $[Y]$ generates $H$; and moreover, $f$ sends the connected component of $\operatorname{Stab}(X)$ to the connected component of $\operatorname{Stab}(Y)$.
        \end{remark}

        \subsection{Characterization of Simple Classes}

        Lemma \ref{L: weak existence of factorization} is slightly misleading, because if we consider arbitrary subvarieties (rather than generating subvarieties with finite stabilizer), there can be exotic simple classes. For example, one can find a semiabelian variety $G$, and a semiabelian subvariety $H\leq G$, such that $H$ has no complement in $G$ (i.e. there is no $H'$ with $G$ isogenous to $H\times H'$). In this case, if $H$ and $G/H$ are simple as semiabelian varieties, then $(G,[H])$ is simple as a semiabelian pair.

        In contrast, in the case of abelian varieties, the usual simple factorization lets us give a nicer characterization. Note that if $G$ and $H$ are isogenous semiabelian vareities, and one is an abelian variety, then so is the other. Thus, we can speak of \textit{weak isogeny classes of abelian pairs.}

        \begin{lemma}\label{L: simple characterization}
            Let $[G,X]$ be a weak isogeny class of abelian pairs. Then $[G,X]$ is simple if and only if one of the following holds:
            \begin{enumerate}
                \item $G$ is a simple abelian variety and $X=G$.
                \item $G$ is a simple abelian variety and $X$ is a single point.
                \item $(G,+,X)$ is a full relic.
            \end{enumerate}
            \end{lemma}
            \begin{proof} First suppose $(G,[X])$ is simple. 
            \begin{itemize}
                \item Case 1: $X$ has infinite stabilizer. Then let $H$ be the connected component of the stabilizer, and let $H'\leq G$ be a complement of $H$ (see, e.g., \cite[II, \S 2, p.30]{LangAV}, so that the sum map $H\times H'\rightarrow G$ is an isogeny. The projection $H'\rightarrow G/H$ is an isogeny. By Lemma \ref{L: isogeny}, there is a semiabelian pair $(H',[Y])$ projecting to $(G/H,[X/H])$. Then $H+Y=X$, so that $[G,X]$ factors into $[H,H]$ and $[H',Y]$. By simplicity, $H=G$, and thus also $X=G$. So our original class is just $[G,G]$. A factorization of this class corresponds to a usual factorization of abelian varieties. So by simplicity of the class, $G$ is simple, and we have case (1) of the lemma.
                \item Case 2: $X$ is contained in a coset of a proper abelian subvariety $H<G$. As in case 1, let $H'\leq G$ be a complement of $H$. Let $a\in H'$ be an element mapping to the coset containing $X$ under the isogeny $H'\rightarrow G/H$. Then $[G,X]$ factors as the product of $[H,X]$ and $[H',\{a\}]$. By simplicity, $H$ is the trivial subgroup. Thus our original class is just $[G,X]$ with $X$ a single point. Again, a factorization of $[G,X]$ now corresponds to a factorization of $G$ into simple abelian varieties, and by simplicity, it follows that $G$ is simple.
                \item Case 3: $X$ has finite stabilizer and generates $G$. Then by Theorem \ref{T: main} (and simplicity), $(G,+,X)$ is full.
            \end{itemize}
            
            Now we show the converse. If either (1) or (2) holds (i.e. $G$ is a simple abelian variety and $X$ is either $G$ or a point), then simplicity is clear. Now assume (3) holds. Again, since (1) and (2) are clear, we may also assume (1) and (2) fail. Then simplicity follows by Corollary \ref{C: big equivalence}.
            \end{proof}

            \begin{remark} Note that conditions (1)-(3) of Lemma \ref{L: simple characterization} are mutually exclusive (except if $G$ is the trivial group, in which case (1) and (2) coincide). The fact that (3) is incompatible with (1) and (2) is because in cases (1) and (2), the relic $(G,+,X)$ is a pure abelian group.
            \end{remark}

            \begin{remark}
                A corollary of Lemma \ref{L: simple characterization} is that conditions (1)-(3) are weak isogeny invariant. This is most interesting for (3), as for (1) and (2) it is clear.
            \end{remark}

            Before moving on to the main result, we point out that Lemma \ref{L: simple characterization} allows for a complete answer to Question \ref{Q: main question}:

            \begin{corollary}\label{C: determines characterization} Suppose $G$ is an abelian variety and $X$ is a closed irreducible subvariety. Then the following are equivalent:
            \begin{enumerate}
                \item $X$ geometrically determines $G$.
                \item $(G,[X])$ is simple and $X$ is not a single point or all of $G$.
                \item $G$ and $X$ satisfy the hypotheses of Theorem \ref{T: main}.
            \end{enumerate}
            \end{corollary}
            \begin{proof}
                $(1)\rightarrow(2)$: Suppose $X$ geometrically determines $G$. Then Lemma \ref{L: determines implies simple} gives that $(G,[X])$ is simple, and by Lemma \ref{L: simple characterization}, it suffices to show that $X$ is not a point or all of $G$. But if $X$ is a point or all of $G$, then an automorphism of $(G,+,X)$ is just an abstract group automorphism of $(G,+)$ (possibly required to fix one point). It is then easy to build exotic group automorphisms and conclude that $X$ does not geometrically determine $G$. We leave the details to the reader, as this is likely well-known. The main ingredient is the fact that over an algebraically closed field, the rank of an abelian variety (i.e. the dimension of its tensor product with $\mathbb Q$ as a vector space) is always infinite.
                
                $(2)\rightarrow(3)$: Suppose $(G,[X])$ is simple and $X$ is not a point or all of $G$. It follows easily that $X$ has finite stabilizer and generates $G$, and this shows that the hypotheses of Theorem \ref{T: main} are met. For example, suppose $X$ is stabilized by a non-trivial abelian subvariety $H_1\leq G$ (the case that $X$ fails to generate $G$ is similar). Let $H_2$ be an orthogonal complement of $H_1$ in $G$. Then one can factor $(G,[X])$ using $H_1$ and $H_2$. By simplicity we get $H_1=G$ -- thus $X$ is stabilized by all of $G$, and thus $X=G$, a contradiction.
                
                $(3)\rightarrow(1)$: This is the statement of Theorem \ref{T: main}.
            \end{proof}

        \subsection{The Unique Factorization Theorem}

        We now give our most general result on factorizations -- namely, we show that for abelian varieties, the factorization into simple classes is unique.

        \begin{theorem}\label{T: unique factorization}
            Let $G$ be an abelian variety, and $X$ a closed irreducible subvariety. Then the factorization of the weak isogeny class $[G,X]$ into simple classes is unique.
        \end{theorem}
        \begin{proof}
            Consider a factorization of the pair $(G,[X])$ into $\{(G_i,[X_i])\}$. By Lemma \ref{L: simple characterization}, each $X_i$ is either equal to $G_i$ or has finite stabilizer. Let $S$ be the sum of those $G_i$ for which $X_i=G_i$. It follows that $S$ is the connected component of $\operatorname{Stab}(X)$. In particular, $S$ is determined up to isogeny by the class $[G,X]$ (see Remark \ref{R: conditions isogeny invariant}). Moreover, the factorization of $S$ into the $G_i$ is unique up to isogeny, by the usual unique factorization of abelian varieties.

            Now let $T$ be the sum of those $G_i$ so that $X_i$ has finite stabilizer (so $T$ is a complement of $S$), and let $X_T$ be the sum of the corresponding $X_i$. Then the projection gives an isogeny $(T,[X_T])\rightarrow(G/S,[X/S])$, which means we reduce the problem to uniqueness of factorization for $[T,X_T]$.

            Thus, from now on, we assume $S$ is trivial, and thus each $X_i$ has finite stabilizer.

            Now again by Lemma \ref{L: simple characterization}, each $X_i$ is either a point or generates $G_i$. Let $R$ be the sum of those $G_i$ for which $X_i$ generates $G_i$, and let $P$ be the sum of those $G_i$ such that $X_i$ is a point. Let $X_R$ and $X_P$ be the sums of the corresponding $X_i$. So $X_P$ is one point.
            
            Note, then, that $R$ is the smallest abelian subvariety of $G$ such that $X$ is contained in one coset of $R$. Thus $R$ is determined by the pair $(G,[X])$, and moreover it is easy to see that this definition of $R$ is isogeny invariant. Moreover, the translation class $[X_R]$ is determined as well, because $X_R$ is just a translate of $X$. So, put together, the class $[G_R,X_R]$ is determined by the class $[G,X]$. So a factorization of $[G,X]$ consists of a factorization of $[G_R,X_R]$ together with a factorization of $[G_P,X_P]$. The factorization of $[G_R,X_R]$ is unique by Lemma \ref{L: uniqueness}, and the factorization of $[G_P,X_P]$ is just a factorization of the abelian variety $G_P$ into simple factors, so is unique as well.
        \end{proof}

\appendix

\section{Stability for the Algebraic Geometer}\label{S: crash course}

The proofs of Theorems \ref{T: main} and \ref{T: general} are really arguments in \textit{stability} -- a part of model theory that can be viewed as aimed at developing abstract notions of independence and dimension. The idea is that the structure $(A(K),+,V(K))$ (as in the statement of Theorem \ref{T: main}) has its own `internal' version of dimension theory, parallel to the `external' dimension theory coming from the field $K$. On some small level, one could think of this internal theory as `algebraic geometry over $(A(K),+,V(K))$'. The proof, then, essentially plays the two dimension theories against each other, with the goal of showing that they are really the same.

In fact, important parts of stability were consciously modeled after the dimension theory of algebraic varieties, and many notions from stability can be seen as abstractions of specific notions in algebraic geometry. Since this paper may be of interest to readers in algebraic geometry, we now attempt to explain the main concepts of stability in geometric terms. \textbf{For the model theorist: in precise terms, we will be describing the behavior of $U$-rank in theories of finite Morley rank}. This is by no means a text on stability, and may from time to time prefer clarity over precision. The goal is for the reader to have some intuition for the words that appear repeatedly in the rest of the paper. For a more thorough account of these terms, see the first chapters of \cite{PillayBook} (mainly the first two chapters).  

\subsection{The basic setup}

Throughout, we work in a saturated model $\mathcal M=(M,...)$ of a complete first-order theory of finite Morley rank. This means we have a first-order language (such as $(+,\times)$), and we have chosen a very large model of it (such as an uncountable algebraically closed field), with underlying set $M$. The `finite Morley rank' clause is a strong form of stability. In particular, we remark that any uncountable algebraically closed field, as well as any structure $(A(K),+,V(K))$ (as in Theorem \ref{T: main}) over an uncountable algebraically field, is an example of a model with these properties.

Throughout, the reader should keep in mind the main example, where $\mathcal M=(K,+,\times)$ is an (uncountable) algebraically closed field. We refer to this example as `ACF'. Our goal is to describe various notions from stability by saying what they look like in ACF.

\subsection{Definable Sets} The basic object of model theory is a \textit{definable set}. This is any subset of $M^n$ singled out by a formula in the relevant language, potentially over parameters. For example, in ACF, the line $\{(x,y):y=ax+b\}$ is definable for any $a$ and $b$. If we wish to care about parameters, we may call this line `$(a,b)$-definable' or `definable over $(a,b)$'. In a similar way, if $V\subset K^n$ is any affine algebraic set, defined by polynomials over a field $F\leq K$, then $V$ is $F$-definable. A theorem of Tarski (equivalent to one of Chevalley) says that this is essentially exhaustive -- in ACF, every definable set is a finite Boolean combination of affine algebraic sets (such sets are called \textit{constructible}).

\subsection{Interpretable Sets} In modern years, it is commonplace to augment a structure by new `sorts' representing definable quotient objects (i.e. quotients of definable sets by definable equivalence relations). Quotients like this are called \textit{interpretable sets}. They are treated, for all intents and purposes, like new definable sets, and one can check that `formally adjoining' them to a language does not cause any problems. The union of \textit{all} of the new sorts is called $M^{eq}$ (which serves as a replacement for the `universe' $M$).

In ACF, a canonical example of an interpretable set is a projective algebraic set (i.e. the set of $K$-points of a projective variety). For example, the $K$-points of projective $n$-space are formed by quotienting $K^{n+1}-\{(0,...,0)\}$ (a definable set) by the equivalence relation given by scalings (a definable equivalence relation). And in general, there is a way to view every interpretable set as a \textit{projective constructible set} -- a finite Boolean combination of sets of this form.

\subsection{Complete Types} Definable and interpretable sets are analogous to \textit{classical} algebraic sets -- not to varieties viewed as schemes. The model-theoretic analog of a (\textit{scheme-theoretic}) irreducible variety over a subfield $F\leq K$ is a \textit{complete type over a subset}. This consists of a \textit{sort} $S$ (an ambient space such as projective or affine space of a certain dimension), a \textit{parameter set} $A\subset M^{eq}$ (e.g. the field $F$), and a complete description $p$ of an element of $S$ over $A$ (meaning a choice of which formulas with parameters in $A$ should be true or false). A point $a\in S$ satisfying all the formulas in $p$ is called a \textit{realization of $p$}. One says things like $a\models p$ (read "$a$ models $p$" or "$a$ satisfies $p$")  and $p=\tp(a/A)$. Note that by convention, one only considers types over \textit{small sets} -- sets $A\subset M^{eq}$ with $|A|<|M|$. This ensures that:
\begin{enumerate}
    \item Every complete type has a realization (this is the actual meaning of the term `saturated').
    \item Two points $a$ and $b$ realize the same type over $A$ if and only if they are automorphism conjugate over $A$ (i.e. there is $\sigma\in\operatorname{Aut}(\mathcal M)$ fixing $A$ pointwise and satisfying $\sigma(a)=b$).
\end{enumerate}

(1) and (2) above essentially follow from the fact that $M$ is very large.

In ACF, there is a bijective correspondence between complete types in $K^n$ over $A\subset K$ and irreducible affine varieties over the field $F=\langle A\rangle$ generated by $A$ (and the same holds in projective space for projective varieties; note that irreducible here means over $F$, not geometrically irreducible). This works as follows: given a complete type $p=\tp(a/A)$ in $K^n$, the type $p$ will specify which polynomials from $F[x_1,...,x_n]$ send $a$ to 0. The set of such polynomials is a prime ideal in $F[x_1,...,x_n]$ (and every prime ideal arises in this way). The injectivity of this map follows from Tariski's (and Chevalley's) theorem above.

Note that given a complete type $p$ over $A$, mapping to an irreducible variety $V$ over $F=\langle A\rangle\leq K$, the realizations of $p$ are exactly the points of $V(K)$ lying over the generic point of $V$ over $F$. We call such points \textit{generic points of $V(K)$ over $A$}. There are plenty of non-generic points of $V(K)$ over $A$ as well, and these do not realize the type corresponding to $V$ over $A$ -- each such point realizes the type corresponding to a proper subvariety of $V$. Roughly speaking, this is how model theory deals with the lack of a Zariski topology in abstract settings (we care about generic points of $V(K)$ only, and not the points in their closure). 

\subsection{Non-Forking Extensions} The main technical tool of stability is an analog of the base change operation. The idea is as follows: suppose we have a type $p$ over a field $F$, corresponding to a variety $V$ over $F$. Given an extension $L$ of $F$, how can we identify abstractly the types over $L$ corresponding to the (components of the) base change of $V$ to $L$? The realizations of these types are the generic elements of $V(K)$ over $L$ -- i.e. those realizations of $p$ which remain generic in $V(K)$ over the bigger field $L$. So to make sense of this in the abstract context, one needs to decide abstractly which realizations of a type remain `generic' over a larger parameter set.

The correct definition was worked out by either Shelah or Morley, depending on the level of generality. In our abstract structure $\mathcal M$, one is given a complete type $p$ over $A$, and a new parameter set $B\supset A$. There are, in general, many complete types $q$ over $B$ extending $p$. Shelah (or Morley) identified a distinguished non-empty finite subset of these, called the \textit{non-forking extensions of $p$ over $B$}. We only explain the definition below in the special case of ACF (in general, even in ACF-relics, the definition is much more technical). Non-forking extensions are thought of as generic/free extensions. In ACF, if $p$ corresponds to a variety $V$ over $F=\langle A\rangle$, and $L=\langle B\rangle$, then the non-forking extensions of $p$ over $B$ correspond to the $L$-irreducible components of the base change of $V$ over $L$. Recall that there may be more than one such component (i.e. irreducibility is not preserved over field extensions).

\subsection{Stationarity and Algebraic Closure} Recall that a variety $V$ over $F$ is \textit{geometrically irreducible} if $V$ remains irreducible over $F^{alg}$, the algebraic closure of $F$. In this case, $V$ remains irreducible over any extension of $F$.

The same works in the abstract setup. Given $A\subset M^{eq}$, one defines $\acl(A)$ (the algebraic closure of $A$) to be the set of $b\in M^{eq}$ such that $\tp(b/A)$ has only finitely many realizations. (In ACF, restricted to the `home sort' $K$, this is the same as the field $\langle A\rangle^{alg}$.) If $\acl(A)=A$ then $A$ is called \textit{algebraically closed}.

Now given a complete type $p$ over $A$, one shows the following:

\begin{enumerate}
\item There are only finitely many extensions of $p$ over $\acl(A)$, and moreover all of them are non-forking.
\item If $A$ is algebraically closed, then for any $B\supset A$, there is only one non-forking extension of $p$ over $B$.
\end{enumerate}

One now says that $p$ is \textit{stationary} if for any $B\supset A$, $p$ has only one non-forking extension over $B$.  This extension is called $p_B$. Equivalently, $p$ is stationary if $p$ has a unique extension to $\acl(A)$. So stationary types are the analog of geometrically irreducible varieties.

\subsection{Dimension Theory} In ACF, if $p$ is a complete type over $A$ corresponding to the variety $V$, then a \textit{forking extension} of $p$ (i.e. not non-forking) corresponds to a subvariety of $V$ of smaller dimension (base changed to a bigger field). This allows for the usual definition of dimension (which model theorists call \textit{$U$-rank}). Precisely, one defines $U(p)$ as the supremum of the lengths of all \textit{forking chains} $p=p_0\geq p_1...\geq p_n$ of complete types -- where `forking chain' means that each $p_{i+1}$ is a forking extension of $p_i$. If $a\models p$ and $U(p)=r$, one also says $U(a/A)=r$.  One can then show:

\begin{enumerate}
\item $U(p)$ is always finite.
\item $U(p)=0$ if and only if $p$ has only finitely many realizations. Equivalently, $U(a/A)=0$ if and only if $a\in\acl(A)$.
\end{enumerate}

\subsection{Minimal Types} A complete type $p$ over $A$ is called \textit{minimal} if $p$ is stationary and $U(p)=1$. This is the analog of a geometrically irreducible curve. Just as curves are paramount to algebraic geometry, minimal types are paramount to stability theory: indeed, much of the structural analysis of theories of finite Morley rank works by trying to classify minimal types. The key notion is \textit{Zilber's trichotomy} (see below), which is an old (and ultimately false) conjecture proposing such a classification into three families. A true instance of Zilber's trichotomy is the key ingredient in the background of our proof.

One can also characterize minimal types as follows: $p$ is minimal if and only if every interpretable set $X$ intersects the realizations of $p$ in a finite or cofinite set (i.e. $p$ cannot be split into two infinite pieces). This is analogous to the fact that a subvariety of an irreducible curve is either dense open or of dimension 0.

\subsection{Additivity and Independence}

The most important feature of $U$-rank is the so-called \textit{additivity formula}: for any $a$, $b$, and $A$, one has $$U(a,b/A)=U(a/A)+U(b/Aa)=U(b/A)+U(a/Ab)$$ (here $Aa$ and $Ab$ are shorthand for $A\cup\{a\}$ and $A\cup\{b\}$, respectively). This is a sort of fiber-dimension theorem. The algebraic geometry analog is the fact that the dimension of the generic fiber of a dominant map $f:X\rightarrow Y$ is just $\dim(X)-\dim(Y)$ (the `dominant map' in question is the projection $\tp(a,b/A)\rightarrow\tp(a/A)$ (or $\tp(b/A)$), where e.g. the fiber above $\tp(a/A)$ is $\tp(b/Aa)$).

One now says that $a$ and $b$ are \textit{independent over $A$} if $U(a/Ab)=U(a/A)$. By additivity, this is equivalent to $U(a,b/A)=U(a/A)+U(b/A)$ and also to $U(b/Aa)=U(b/A)$. In ACF, this says that the variety corresponding to $\tp(ab/A)$ is the product of the varieties corresponding to $\tp(a/A)$ and $\tp(b/A)$. Of course, one can extend this to $n$ points, and say that $a_1,...,a_n$ are independent over $A$ if $U(a_1,...,a,_n/A)$ is the sum of the $U(a_i/A)$ -- this is analogous to taking the product of the $n$ varieties corresponding to the $\tp(a_i/A)$.

Note that independent points always exist: that is, given any $a$, $b$, and $A$, there is some $b'\models\tp(b/A)$ with $a$ and $b'$ independent over $A$ -- one chooses $b'$ realizing any non-forking extension of $\tp(a/A)$ over the larger set $Ab$.

\subsection{Orthogonality and Internality} This point is about two notions which act trivially in algebraic geometry, and thus are unique to the abstract setup. On the other hand, these notions are crucial for the general theory, and are among the most essential tools in our proof.

Over an algebraically closed field, every irreducible variety of positive dimension admits a dominant rational map to the affine line. It follows that given two such varieties $X$ and $Y$, there is always a non-trivial definable relation between the $K$-points of $X$ and $Y$ (namely the relation $f(x)=g(y)$, where $f$ and $g$ are rational functions on $X$ and $Y$, respectively). In contrast, in the abstract setting, two definable sets might `have nothing to do with each other'. For example, $\mathcal M$ could just be the disjoint union of two algebraically closed fields $K_1$ and $K_2$, with no relation at all between them. In this case, we say that $K_1$ and $K_2$ are \textit{orthogonal}. Of course, there are less trivial examples -- a good deal of the (quite active) work on the model theory of differential equations involves determining when the solution sets of two equations are orthogonal.

Formally speaking, let $p$ and $q$ be stationary types over $A$ and $B$, respectively. We say $p$ and $q$ are \textit{orthogonal} if for any $C\supset A\cup B$, any realization of $p_C$ is independent from any realization of $q_C$ over $C$. One can extend this by replacing either $p$ or $q$ (or both) with an interpretable set. For example, $p$ and $X$ are orthogonal if $p$ is orthogonal to every stationary type lying in $X$.

An opposite notion is \textit{internality}. Roughly speaking, $p$ is \textit{internal to $q$} (resp. \textit{almost internal to $q$}) if $p$ `can be described entirely (resp. up to finite covers) using $q$' -- formally, potentially modulo a finite set of parameters, every realization of $p$ is definable\footnote{An element $b$ is \textit{definable} over $A$ if $b\in f(A)$ for some $\emptyset$-definable function $f$. Equivalently, if $b$ is the unique realization of $\tp(b/A)$.} (resp. algebraic) over a finite tuple of realizations of $q$. As above, one can get similar notions replacing $p$, $q$, or both with an interpretable set. Geometrically, one thinks of internality as saying there is a rational map $r\rightarrow p$ where $r$ is a type inside $q^n$; while almost internality allows $r$ to be a finite cover of a type in $q^n$. Note that internality and almost internality have the obvious transitivity properties. In ACF, any two types (resp. interpretable sets) are internal to each other.

There is a certain calculus of orthogonality and internality, coming from the characterization of minimal types. This is technical, but also intuitive. For example:
\begin{enumerate}
\item Any two minimal types are either orthogonal or almost internal to each other -- and the latter case induces an equivalence relation on minimal types. 
\item More generally, if $p$ is a minimal type, and $X$ is either a type or an interpretable set, then $p$ is either orthogonal to $X$ or almost internal to $X$ (and cannot be both).
\item If $X$ and $Y$ are either types or interpretable sets, then $X$ and $Y$ are non-orthogonal if and only if there is a minimal type almost internal to both.
\end{enumerate}

\subsection{Zilber's Trichotomy} In the abstract setup, there might not be a very rich structure on a minimal type. For example, $\mathcal M$ might be a pure infinite set with no structure, in which case $M$ itself is (the set of realizations of) a `trivial' minimal type. Other minimal types arise in linear algebra as generic types of affine spaces (equipped only with `affine linear' structure). A good deal of work has gone into sorting out the possible complexity levels of minimal types -- the theme being \textit{Zilber's Trichotomy}.

As stated above, Zilber's trichotomy was a proposed classification of all minimal types into three classes. We will not define these explicitly. The key point is that Zilber isolated a dividing line among minimal types, breaking them into those of `trivial or linear complexity' (called \textit{locally modular}) and those of `non-linear complexity' (called \textit{non-locally modular}). (The `trichotomy' refers to trivial, linear, and non-linear.) In ACF, every minimal type is non-locally modular. Zilber's original conjecture (roughly) said that non-local modularity could only happen in ACF.

For now, we only mention that the equivalence relation on minimal types coming from almost internality/non-orthogonality refines Zilber's suggested classification (that is, two minimal types of different levels must be orthogonal).

\subsection{Groups, Connected Components, and Stabilizers} A \textit{definable group} is a group whose underlying set and group operation are definable (more generally, one could take \textit{interpretable groups}, whose underlying set and operation are \textit{interpretable}). Definable groups are the analog of algebraic groups. In fact, in ACF, every definable (or even interpretable) group is isomorphic to an algebraic group via an interpretable isomorphism.

If $G$ is a definable group, and $p$ is a stationary type in $G$, then $p$ has a \textit{stabilizer}, analogous to the stabilizer of a closed (irreducible) subvariety of an algebraic group. The model theoretic version of the stabilizer is more delicate. Recall that our `subvarieties' are not closed (they are analogous to generic parts of closed sets). Because of this, the proper definition of `stabilizing' is actually `generic stabilizing'. Formally, if $p$ is a type over $A$, then $\operatorname{Stab}(p)$ is the set of $g$ such that $g$ preserves the set of realizations of the non-forking extension $p_g$. The stabilizer is always a definable subgroup of $G$.

Over an algebraically closed field, the connected component of the identity in an algebraic group $G$ is an irreducible normal algebraic subgroup; its cosets form both the irreducible components of $G$ and the connected components of $G$. This works also in the abstract setup. Suppose $G$ is a definable group, whose underlying set and operation are definable over $A$. Then one can isolate a non-empty finite set of `generic types' of $G$ -- complete (stationary) types over $\acl(A)$ inside of $G$ of maximum $U$-rank. The realizations of these types are called the \textit{generic} elements of $G$ over $A$. One shows that $G$ acts transitively on its generic types, and all of the generic types have the same stabilizer, say $G^0\leq G$. $G^0$ is always normal, and is the smallest definable subgroup of $G$ of finite index. It follows that $G^0$ has exactly one generic type, and the other generic types are just the translates of this type to the cosets of $G^0$. We call $G^0$ the \textit{connected component of $G$}. If $G=G^0$ then $G$ is \textit{connected}. This is equivalent to $G$ having only one generic type. In ACF, $G^0$ coincides with the usual connected component of the identity, and the generic types are just the complete types corresponding to the cosets of $G^0$.

\subsection{Stable Embededness} Finally, we discuss a subtle but useful fact about stable theories that plays an essential role in our proof. Often in model theory, one needs to worry about what parameters are needed to define a specific set. A \textit{stably embedded} set is a definable (or interpretable set) $X$ such that, to define subsets of $X^n$, one only needs parameters from $X$ itself. Formally, $X$ is \textit{stably embedded over $A$} if $X$ is $A$-definable, and every definable subset of every $X^n$ is definable over $A\cup X$.

In stable theories, stable embeddedness is automatic: if $X$ is $A$-definable, then $X$ is stably embedded over $A$. The ACF analogy is the fact that, given a fixed algebraic family $X\rightarrow T$ of varieties, each fiber $X_t$ (for $t$ a $K$-rational point of $T$) is determined by finitely many points in $X_t(K)$ (for example, a polynomial $P\in K[x]$ of degree $d$ is determined by its action on $d+1$ distinct points).

            \bibliography{references}
        \bibliographystyle{plain}
    \end{document}